\documentclass[11pt,reqno]{amsart}
\usepackage{amsmath,amsfonts,amsthm,amssymb,amsxtra}
\newcommand\version{October 9, 2007} 

\usepackage{amsxtra, amssymb, mathrsfs, pstricks}

\newtheorem{theorem}{Theorem}[section]
\newtheorem{proposition}[theorem]{Proposition}

\newtheorem{lemma}[theorem]{Lemma}
\newtheorem{corollary}[theorem]{Corollary}

\theoremstyle{definition}

\newtheorem{assumption}[theorem]{Assumption}
\newtheorem{definition}[theorem]{Definition}
\newtheorem{example}[theorem]{Example}

\theoremstyle{remark}

\newtheorem{remark}[theorem]{Remark}

\numberwithin{equation}{section}

\newcommand{\A}{\mathcal{A}}
\newcommand{\B}{\mathfrak{B}}
\renewcommand{\b}{\mathfrak{b}}

\newcommand{\cl}{{\rm cl}}

\newcommand{\Dirichlet}{\mathscr{D}}

\renewcommand{\H}{H}

\newcommand{\loc}{{\rm loc}}
\newcommand{\N}{\mathbb{N}}
\newcommand{\Neumann}{\mathcal{N}}
\newcommand{\qf}{\mathfrak{a}}
\newcommand{\R}{\mathbb{R}}
\newcommand{\U}{\mathcal{U}}
\newcommand{\Z}{\mathbb{Z}}

\DeclareMathOperator{\dist}{dist}

\DeclareMathOperator{\spec}{spec}

\DeclareMathOperator{\tr}{tr}

\title[Eigenvalue estimates --- \version]{Eigenvalue estimates for Schr\"odinger operators on metric trees}

\author{Tomas Ekholm}
\address{Tomas Ekholm, Centre for Mathematical Sciences, Lund
  University, Box 118, 22100 Lund, Sweden} 
\email{tomase@maths.lth.se}

\author{Rupert L. Frank}
\address{Rupert L. Frank, Department of Mathematics, Fine Hall,
  Princeton University, Princeton, NJ 08544, USA} 
\email{rlfrank@math.princeton.edu}

\author{Hynek Kova\v r\'{\i}k}
\address{Hynek Kova\v r\'{\i}k, Department of Mathematics, Stuttgart
  University, Pfaffenwaldring 57, 70569 Stuttgart, Germany} 
\email{hynek.kovarik@mathematik.uni-stuttgart.de} 

\begin{document}

\begin{abstract}
We consider Schr\"odinger operators on regular metric trees and prove
Lieb-Thirring and Cwikel-Lieb-Rozenblum inequalities for their
negative eigenvalues. The validity of these inequalities depends on
the volume growth of the tree. We show that the bounds are valid in
the endpoint case and reflect the correct order in the weak or strong
coupling limit. 
\end{abstract}


\thanks{\copyright\, 2007 by the authors. This paper may be  
reproduced, in
its entirety, for non-commercial purposes.}

\maketitle


\section{Introduction}

It is well known that the moments of negative eigenvalues of the
Schr\"o\-din\-ger operator $-\Delta-V$ in $L_2(\R^d)$ can be estimated in
terms of the classical phase space volume. Namely, the Lieb-Thirring
inequality states that the bound
\begin{equation} \label{LT-classical}
\tr\, (-\Delta-V)_-^{\gamma} \, \leq \, L_{\gamma,d}\, \int_{\R^d}\,
V^{\gamma+\frac d2}_+ \, dx\
\end{equation}
holds true for any potential $V$ if and only if
\begin{equation} \label{assumptions-classical}
\gamma \geq \frac 12 \, \, \,  \text{if }\, \,   d = 1, \quad
\gamma > 0 \, \, \,  \text{if } \, \,  d = 2,
\quad \gamma \geq 0\,
\, \,  \text{if } \, \,  d \geq 3\, .
\end{equation}
Here $x_\pm:=\max\{0,\pm x\}$ denotes the positive and negative part
of $x$. Inequality \eqref{LT-classical} is due to Lieb and Thirring
\cite{LT} and, in the endpoint cases, to Cwikel \cite{Cw}, Lieb
\cite{L}, Rozenblum \cite{R} and Weidl \cite{W}. We refer to \cite{LW}
and \cite{H} for recent reviews on this topic.

Our main objective is to establish the analog of \eqref{LT-classical}
for Schr\"odinger operators on metric trees. A (rooted) metric tree
$\Gamma$ consists of a set of vertices and a set of edges, i.e.,
segments of the real axis which connect the vertices. We assume that
$\Gamma$ has infinite height, that is, it contains points at
arbitrary large distance from the root. We define the Schr\"odinger
operator formally as 
\begin{align*}
-\Delta_\Neumann -V\, \quad \text{in}\quad L_2(\Gamma)
\end{align*}
with Kirchhoff matching conditions at the vertices and a Neumann
boundary condition at the root of the tree. 

Metric trees represent a special class of so called quantum graphs,
which recently have attracted great interest; see, e.g., \cite{BCFK,
  KoS, Ku1, Ku2} for extensive bibliographies about this subject. Many
works devoted to quantum graphs concern questions about self-adjoint
extensions, approximation by thin quantum wave guides and direct or
inverse scattering properties of the Laplace operator on graphs, see
the references above and also \cite{EP, KuS}. Various functional
inequalities for the Laplacian on metric trees have been established
in 
\cite{EHP,NS2}. However, much less attention has been paid, with the
exception of \cite{NS1}, to the classical question of finding appropriate
estimates, similar to \eqref{LT-classical}, on the discrete spectrum
of Schr\"odinger operators on metric trees. As we shall see, the
interplay between the spectral theory and the mixed dimensonality of a
tree makes this a fascinating problem. 

Our main result concern \emph{regular} metric trees, that is, trees
which are symmetric with respect to the distance from the root; see
Subsection \ref{prelim} for a precise definition. We shall show that
the validity of a suitable analog of \eqref{LT-classical} is
characterized by the global branching of the tree $\Gamma$.  The
latter is expressed by the \emph{branching function}
$g_0(t):=\#\{x:|x|=t\}$ which counts the number of points of $\Gamma$
as a function of the distance from the root. The function $g_0$ is
clearly non-decreasing. Depending on its growth we may split the trees
into two classes according to whether the integral 
\begin{equation} \label{integral}
\int_0^\infty\, \frac{dt}{g_0(t)}
\end{equation}
is finite (\emph{transient trees}) or infinite (\emph{recurrent
  trees}). It turns out that in the former case, the corresponding
Lieb-Thirring inequality holds for all values $\gamma\geq 0$. For
$\gamma=0$ this is an estimate on the number of negative eigenvalues
in terms of an integral of the potential, usually called a
Cwikel-Lieb-Rozenblum inequality. On the other hand, if the integral
\eqref{integral} is infinite, then Lieb-Thirring inequalities do not
hold for values of $\gamma$ which are smaller than some critical value
$\gamma_{\rm min}>0$. In order to determine the value of $\gamma_{\rm
  min}$ we use the notion of the {\it global dimension} of a metric
tree, see Definition \ref{dim}. This dimension is equal to $d\geq 1$
if the branching function $g_0$ has a power-like growth at infinity with
power $d-1$. We emphasize that in contrast to the Euclidean case, $d$
need not be an integer.  

For regular metric trees $\Gamma$ with global dimension $d$ and
Schr\"odinger operators with symmetric potentials $V$ we shall prove
Lieb-Thirring inequalities of the form 
\begin{equation} \label{LT-tree}
\tr\, (-\Delta_\Neumann- V)_-^{\gamma} \, \leq \, C\,
\int_{\Gamma}\, V^{\gamma+\frac{1+a}{2}}_+ \, g_0^{\frac{a}{d-1}}\,
dx\, , \quad a\geq 0\, .
\end{equation}
The allowed values of $\gamma$ are determined by the parameter $a$ and
by the global dimension $d$ of $\Gamma$, see Theorem
\ref{th:mainTree}. For $a=0$ the weight in the integral on the right
hand side disappears and the inequality is very similar to its
Euclidean version \eqref{LT-classical}. Both sides then share the same
growth in the strong coupling limit, see Remark \ref{strong}
below. On the other hand, it requires the exponent $\gamma\geq 1/2$
and does not capture the fact that even smaller moments can be
estimated for larger values of $d$. This motivates the inequality
\eqref{LT-tree} with different choices of $a$. As a consequence of our
result, the smallest value of $\gamma$ such that \eqref{LT-tree} holds
for some $a\geq 0$ (indeed, for $a=d-1$) is 
\begin{equation} \label{minimum}
\gamma_{\min} = \frac{2-d}{2} \quad 1\leq d <2\, , \quad
\gamma_{\min}=0\quad d >2\, .
\end{equation}
We emphasize that we establish the inequality in these endpoint cases
and that the resulting inequality for $1\leq d<2$ is
\emph{order-sharp} in the weak coupling limit, see Remark
\ref{coupling}. As one may expect by analogy with the Euclidean
situation, the case $d=2$ is somewhat special, since the minimal value
of $\gamma$ is $0$, but the inequality is not valid in the endpoint
case.  

We consider also the case of a homogeneous tree, i.e., a tree where
all edges have equal length and all vertices are of the same
degree. In this case, the function $g_0$ grows exponentially and the
Laplacian $-\Delta_\Neumann$ is positive definite. We prove
Cwikel-Lieb-Rozenblum inequalities for the number of eigenvalues that
a potential $V$ generates below the bottom of the spectrum of
$-\Delta_\Neumann$. 


An important ingredient in our proof of eigenvalue estimates are
one-dimensional Sobolev inequalities with weights. In particular, if
the integral \eqref{integral} is finite, we combine them with a Sturm
oscillation 
argument in order to deduce Cwikel-Lieb-Rozenblum inequalities. This
yields remarkably good bounds on the constants. We believe that our
technique, in particular the duality 
argument in Proposition \ref{duality}, has applications beyond the 
context of this paper. 

As we have pointed out, one of the main motivations for this work is
to understand how the dimensionality of the underlying space is
reflected in eigenvalue estimates. Several results in the literature
can be viewed in this light. If the global dimension of the underlying
space is, in contrast to our situation, \emph{smaller} than the local
dimension, then the eigenvalues are typically estimated by a sum of
two terms. Lieb-Thirring inequalities of this form have been proved by
Lieb, Solovej and Yngvason \cite{LSY} for 
the Pauli operator. The second, non-standard term there corresponds to
states in the lowest Landau level, which are localized in the
plane orthogonal to the magnetic field. A two-term inequality of more obvious
geometric nature was proved by Exner and Weidl \cite{EW} for
Schr\"odinger operators in a waveguide $\omega\times\R$,
$\omega\subset\R^{d-1}$. Here the second term corresponds to the
global dimension, which is one, as opposed to the local dimension
$d$. These two-term estimates are order-sharp both in the weak
coupling regime (where the global dimension is dominant) and in the
strong coupling regime (where the local dimension is dominant). In our
situation, however, the global dimension is \emph{larger} than the
local dimension, and a two-term inequality would neither in the weak
nor in the strong coupling regime be order-sharp. Therefore we propose
\emph{families} of inequalities, which are sharp in different coupling
regimes. This is 
somewhat reminiscent of the family of inequalities proved by Hundertmark
and Simon \cite{HS} for the discrete Laplacian on the lattice $\Z^d$,
where the local dimension is $0$ and the global dimension is $d$.

\subsection*{Acknowledgements}

The authors are grateful to Robert Seiringer and Timo Weidl for
several useful discussions, and to the organizers of the workshop
`Analysis on Graphs' at the 
Isaac Newton Institute in Cambridge for their kind invitation. This
work has been supported by FCT grant SFRH/BPD/ 23820/2005 (T.E.) and
DAAD grant D/06/49117 (R.F.). Partial support by the ESF programme
SPECT (T.E. and H.K.) and the DAAD-STINT PPP programme (R.F.) is
gratefully acknowledged.


\section{Main results and discussions}

\subsection{Preliminaries}
\label{prelim}
Let $\Gamma$ be a rooted metric tree with root $o$. By $|x|$ we denote
the unique distance between a point $x \in\Gamma$ and the root
$o$. Throughout we assume that $\Gamma$ is of infinite height, i.e.,
$\sup_{x\in\Gamma}|x|=\infty$. The branching number $b(x)$ of a vertex
$x$ is defined as the number of edges emanating from $x$. We assume
the natural conditions that $b(x) > 1$ for any vertex $x \neq o$ and
that $b(o)=1$. 

We define the Neumann Laplacian $-\Delta_\Neumann$ as the self-adjoint
operator in $L_2(\Gamma)$ associated with the closed quadratic form
\begin{align}\label{eq:kinetic}
  \int_{\Gamma} |\varphi'(x)|^2 \, dx, \quad \varphi \in \H^1(\Gamma).
\end{align}
Here $\H^1(\Gamma)$ consists of all continuous functions
$\varphi$ such that $\varphi \in \H^1(e)$ on each edge $e$ of $\Gamma$
and 
\begin{align*}
\int_{\Gamma} \left(|\varphi'(x)|^2 + |\varphi(x)|^2\right) \, dx <
\infty. 
\end{align*}
The operator domain of $-\Delta_\Neumann$ consists of all continuous
functions 
$\varphi$ such that $\varphi'(o)=0$, $\varphi \in \H^2(e)$ for each
edge $e$ of $\Gamma$ and such that at each vertex $x \neq o$ of
$\Gamma$ the matching conditions
\begin{align*}
\varphi_-(x)= \varphi_1(x)=\cdots = \varphi_{b(x)}(x)\, , \quad
\varphi'_-(x) = \varphi'_1(x) + \cdots + \varphi'_{b(x)}(x)
\end{align*}
are satisfied. Here $\varphi_-$ denotes the restriction of $\varphi$
on the 
edge terminating in $x$ and $\varphi_j,\, j=1,\ldots,b(x),$ denote the
restrictions of $\varphi$ to the edges emanating from
$x$, see, e.g., \cite{NS1, NS2} for details.

In this paper we are interested in Schr\"odinger operators
$-\Delta_\Neumann-V$ in $L_2(\Gamma)$. Throughout we assume that the
potential $V$ is a real-valued, sufficiently regular function on
$\Gamma$, the positive part of which vanishes at infinity in a
suitable sense. (We shall be more precise below.) In this case the
negative spectrum of $-\Delta_\Neumann-V$ consists of discrete
eigenvalues of finite multiplicities. Our goal is to estimate the
total number of these eigenvalues or, more generally, moments of these
eigenvalues in terms of integrals of the potential $V$.

The starting point of our analysis is

\begin{theorem} 
  \label{nonsym}
  Let $\gamma \geq 1/2$. Then there exists a constant $L_\gamma$ such
  that for any rooted metric tree $\Gamma$ and any $V$, 
  \begin{align} \label{eq:nonsym}
    \tr (-\Delta_\Neumann - V)_-^\gamma
    \leq L_\gamma \,
    \int_{\Gamma} V(x)_+^{\gamma + \frac 12} \, dx.
  \end{align}
\end{theorem}

We emphasize that the constant $L_\gamma$ is independent of
$\Gamma$. This result is clearly analogous to the standard
one-dimensional Lieb-Thirring inequalities. An advantage is its
universality. Moreover, we will see in Subsection \ref{sec:discussion}
below, that the right hand side has the correct order of growth in the
strong coupling limit when $V$ is replaced by $\alpha V$ and
$\alpha\to\infty$. On the other hand, it does not reflect the geometry
of $\Gamma$ at all and it does \emph{not} display the correct behavior
in the weak coupling limit when $V$ is replaced by $\alpha V$ and
$\alpha\to 0$. 
 
The main goal of this paper is to obtain eigenvalue estimates which
take the global structure of $\Gamma$ into account. We shall consider
trees which possess certain additional symmetry properties. Namely, we
impose 

\begin{assumption}\label{ass:regular}
The tree $\Gamma$ is \emph{regular}, i.e., all the vertices at the
same distance from the root have equal branching numbers and all the
edges 
emanating from these vertices have equal length.
\end{assumption}

Let $x$ be a vertex such that there are $k+1$ vertices on the (unique) path
between $o$ and $x$ including the endpoints. We denote by $t_k$ the
distance $|x|$ and by $b_k$ the branching number of $x$. Moreover, we
put $t_0:=0$ and $b_0:=1$. Note that $t_k$ and $b_k$ are only
well-defined for regular 
trees and that these numbers, in the regular case, uniquely determine
the tree. 

We define the \emph{(first) branching function} $g_0 : \R_+ \to \N$ by
\begin{align*}
  g_0(t) := b_0\, b_1\cdots b_k, 
  \quad \text{if} \ t_k < t \leq t_{k+1},
  \quad k\in\N_0.
\end{align*}
Here $\N= \{1,2,3,\ldots\}$ and $\N_0 := \N \cup \{0\}$. Note that
$g_0$ is a non-decreasing function and that $g_0(t)$ coincides with
the number of points $x \in \Gamma$ such that $|x|=t$. The rate of
growth of $g_0$ reflects the rate of growth of the tree $\Gamma$. More
precisely, $g_0$ measures how the surface of the `ball' $\{x\in\Gamma
: |x| < t\}$ grows with $t$. Of great importance in our analysis will
be the fact whether the {\it reduced height} of $\Gamma$, 
\begin{equation}\label{eq:l}
\ell_\Gamma := \int_0^{\infty}\, \frac{dt}{g_0(t)}
\end{equation}
is finite or not.

In addition to Assumption \ref{ass:regular} we shall impose

\begin{assumption}
The function $V$ is \emph{symmetric}, i.e., for any $x\in\Gamma$ the
value $V(x)$ depends only on the distance $|x|$ between $x$ and the
root $o$.
\end{assumption}

With slight abuse of notation we shall write sometimes $V$ instead of
$V(|\cdot|)$.


\subsection{Eigenvalue estimates on trees}

In this subsection we present our main results. We denote by $N(T)$
the number of negative eigenvalues (counting multiplicities) of a
self-adjoint, lower bounded operator $T$. We begin with the case where
the reduced height \eqref{eq:l} is finite. In this case we shall prove 

\begin{theorem}[\textbf{CLR bounds for trees of finite reduced height}]
  \label{finite-rh}
  Let $\Gamma$ be a regular metric tree with $\ell_\Gamma < \infty$
  and let $w:\R_+\to\R_+$ be a positive function such that for some
  $2<q\leq\infty$ 
  \begin{equation}\label{eq:clrass}
    M :=    \sup_{t\geq 0} \left(\int_0^t g_0(s)^{\frac q2}
    w(s)^{-\frac{q-2}2} \,ds\right)^{2/q}
    \int_t^\infty \frac{ds}{g_0(s)} <\infty \ .
  \end{equation}
  Let $p:= q/(q-2)$. Then there exists a constant $N_p(\Gamma,w)$ such
  that
  \begin{equation} 
    \label{eq:clr} 
    N(-\Delta_\Neumann - V) \leq
    N_p(\Gamma,w)\int_{\Gamma} V(|x|)_+^p w(|x|) \, dx\,
  \end{equation}
  for all symmetric $V$. Moreover, the sharp constant in
  \eqref{eq:clr} satisfies
  \begin{equation*}
    N_p(\Gamma,w) \leq (1+p')^{p-1}\left(1+\frac1{p'}\right)^p M^p.
  \end{equation*}
\end{theorem}

By definition, if $q=\infty$ condition \eqref{eq:clrass} is understood
as
\begin{equation*}
  \sup_{t\geq 0} \left(\sup_{0\leq s\leq t} \frac{g_0(s)}{w(s)}\right)
  \int_t^\infty \frac{ds}{g_0(s)}   <\infty\ ,
\end{equation*}
and one has $N_1(\Gamma,w) \leq M$.

In order to give more explicit estimates we assume that the growth of
the branching function is sufficiently regular in the sense of 

\begin{definition} \label{dim}
  A regular metric tree $\Gamma$ has \emph{global dimension} $d\geq 1$
  if its branching function satisfies 
  \begin{align}\label{eq:dim}
    0 < c_1:= \inf_{t\geq0} \frac{g_0(t)}{(1+t)^{d-1}} \leq 
    \sup_{t\geq0} \frac{g_0(t)}{(1+t)^{d-1}} =: c_2 <\infty\, .
  \end{align}
\end{definition}

Obviously, if $\Gamma$ has global dimension $d$, then it has finite
reduced height if and only if $d>2$. In this case Theorem
\ref{finite-rh} implies 

\begin{corollary}\label{clrcor}
  Assume that $\Gamma$ has global dimension $d>2$. Then for any $a\geq
  1$ there exists a constant $C(a,\Gamma)$ such that for any symmetric
  $V$ 
  \begin{equation*}
    N(-\Delta_\Neumann - V)
    \leq C(a,\Gamma)
    \int_{\Gamma} V(|x|)^{\frac{1+a}{2}}\, g_0(|x|)^{\frac{a}{d-1}} \, dx\, . 
  \end{equation*}
\end{corollary}

Next we turn to the case of infinite reduced height
$\ell_\Gamma=\infty$. It is easy to see that Schr\"odinger operators
$-\Delta_\Neumann-V$ on such trees with non-trivial $V\geq 0$ have at
least one negative eigenvalue, no matter how small $V$ is. Hence it is
impossible to estimate the number of eigenvalues from above by a
weighted integral norm of the potential. However, under the assumption
that the tree has a global dimension we can prove estimates for the
\emph{moments} of negative eigenvalues of $-\Delta_\Neumann -
V$. Moreover, we can treat the case $0\leq a< 1$ which was left open
in Corollary \ref{clrcor}. Our result is 

\begin{theorem}[\textbf{LT bounds for trees}] \label{th:mainTree}
  Let $\Gamma$ be a regular metric tree with global dimension $d\geq 1$.
  \begin{enumerate}
  \item 
    Assume that either $1\leq d<2$ and $0\leq a\leq d-1$, or else that
    $d\geq 2$ and $0\leq a <1$. Then for any $\gamma\geq
    \frac{1-a}{2}$ there exists a constant $C(\gamma,a,\Gamma)$ such
    that for any symmetric $V$ 
    \begin{equation} \label{LT-basic} 
      \tr (-\Delta_\Neumann - V)_-^{\gamma} \leq C(\gamma,a,\Gamma)
      \int_{\Gamma} V(|x|)_+^{\gamma + \frac{1 + a}{2}}
      g_0(|x|)^{\frac{a}{d-1}} \, dx. 
    \end{equation}
  \item 
    Assume that either $1\leq d < 2$ and $a>d-1$, or else that $d=2$
    and $a\geq 1$. Then for any $\gamma>(1+a)\, \frac{2-d}{2d}$ there
    exists $C(\gamma,a,\Gamma)$ such that \eqref{LT-basic} holds for
    any symmetric $V$.  
  \item
    Assume that $d > 2$ and that $a\geq 1$. Then for any $\gamma\geq
    0$ there exists $C(\gamma,a,\Gamma)$ such that \eqref{LT-basic}
    holds for any symmetric $V$. 
  \end{enumerate}
\end{theorem}

One can prove that our conditions on $\gamma$ are not only sufficient
but (except for the limiting case in Part (2)) also necessary for the
validity of \eqref{LT-basic}. This is further discussed in Subsection
\ref{sec:discussion}. Part (3) is in fact an immediate consequence of
Corollary \ref{clrcor} and an argument by Aizenman and Lieb
\cite{AL}. It is stated here for the sake of completeness. 

If the branching function $g_0$ grows `very' fast, the Laplacian
$-\Delta_\Neumann$ is positive definite. In this case it is reasonable
not only to estimate the number of negative eigenvalues of
$-\Delta_\Neumann-V$, but also the number of eigenvalues less then the
bottom of the spectrum of $-\Delta_\Neumann$. We carry through this
analysis for a special class of trees. 

A regular metric tree is called \emph{homogeneous} if all the edges
have the same length $\tau$ and if the branching number $b_k=b>1$ is
independent of $k$. Homogeneous trees correspond intuitively to trees
of infinitely large global dimension. By scaling it is no loss of
generality to assume that $\tau=1$. The branching function $g_0$ 
then reads
\begin{equation*}
g_0(t) = b^j, \qquad j< t \leq j+1, \quad j\in\N_0\, .
\end{equation*}
The Laplacian $-\Delta_\Neumann$ (or rather its Dirichlet version) on
a homogeneous tree was studied in \cite{SS}. It follows from the
analysis there that $-\Delta_\Neumann$ is positive definite and its
essential spectrum starts at 
\begin{align*}
\lambda_b = \left(\arccos \frac 1{R_b} \right)^2,
\qquad R_b = \frac{b^{\frac 12}+b^{-\frac 12}}{2}\, .
\end{align*}
We shall prove

\begin{theorem}[\textbf{CLR bounds for homogeneous trees}]\label{homo}
  Let $\Gamma$ be a homogeneous tree with edge length $1$ and
  branching number $b>1$ and let $w:\R_+\to\R_+$ be a positive
  function such that for some $2<q\leq\infty$ 
  \begin{equation*}
    M:= \sup_{t\geq 0} \, (1+t)^{-1} 
    \left(\int_0^t (1+s)^q w^{-\frac{q-2}2} \,ds\right)^{2/q}.
  \end{equation*}
  Let $p = q/(q-2)$. Then there exists a constant $N_p(b,w)$ such that
  \begin{equation}\label{eq:homo}
    N(-\Delta_\Neumann - V-\lambda_b) 
    \leq N_p(b,w) \int_\Gamma V(|x|)_+^p\, w(|x|)\,dx
  \end{equation}
  for all symmetric $V$. Moreover, the sharp constant in
  \eqref{eq:homo} satisfies 
  \begin{equation}
    N_p(b,w) \leq C(b)\, (1+p')^{p-1}\left(1+\frac1{p'}\right)^p M^p
  \end{equation}
  with some constant $C(b)$ depending only on $b$.
\end{theorem}

Choosing $w(t)=(1+t)^a$ we obtain the following strengthening of
Corollary \ref{clrcor}. 

\begin{corollary}
  Let $\Gamma$ be a homogeneous tree with edge length $1$ and
  branching number $b>1$. Then for any $a\geq 1$ there exists a
  constant $C(a,b)$ such that for any symmetric $V$ 
  \begin{equation*}
    N(-\Delta_\Neumann - V-\lambda_b) 
    \leq C(a,b) \int_\Gamma V(|x|)_+^{\frac{1+a}2}\, (1+|x|)^a \,dx\, .
  \end{equation*}
\end{corollary}


\subsection{Discussion}\label{sec:discussion}

In this subsection we discuss the inequality \eqref{LT-basic} and the
conditions for its validity given in Theorem \ref{th:mainTree}. 

\begin{remark}[\textbf{Strong coupling limit}]
  Inequality \eqref{LT-basic} with $a=0$ coincides with \eqref{eq:nonsym},
  \begin{equation*}
    \tr (-\Delta_\Neumann - V)_-^{\gamma} 
    \leq L_\gamma
    \int_{\Gamma} V(|x|)_+^{\gamma + \frac{1}{2}} \, dx,
    \quad \gamma\geq \frac12.
  \end{equation*}
  This inequality reflects the correct behavior in the strong coupling
  limit. Indeed, if $V$ is, say, continuous and of compact support
  then standard Dirichlet-Neumann bracketing \cite[Thm. XIII.80]{RS4}
  leads to the Weyl-type asymptotic formula 
\begin{equation} \label{strong} 
    \lim_{\alpha\to \infty}
    \alpha^{-\gamma-\frac 12} \tr\left(-\Delta_\Neumann - \alpha
      V\right)_-^\gamma 
    = L_{\gamma,1}^{\cl} \int_{\Gamma} V(|x|)^{\gamma + \frac 12}_+ \, dx, 
    \quad \gamma \geq 0,
\end{equation}
  with
  \begin{equation}
    \label{eq:lclass}
    L_{\gamma,1}^{\cl} 
    := \frac{\Gamma(\gamma+1)}{2\sqrt{\pi}\,\Gamma(\gamma + 3/2)}\, .
  \end{equation}
  This shows in particular that \eqref{LT-basic} can \emph{not} hold for $a<0$.
\end{remark}

\begin{remark}[\textbf{Weak coupling limit}] 
  \label{coupling}
  Assume that $\Gamma$ has global dimension $d\in [1,2)$. Inequality
  \eqref{LT-basic} with $a=d-1$, $\gamma=(2-d)/2$ reads  
  \begin{equation*}
    \tr (-\Delta_\Neumann - V)_-^{\frac{2-d}2} \leq
    C\left(\frac{2-d}2,d-1,\Gamma\right) 
    \int_{\Gamma} V(|x|)_+ g_0(|x|) \, dx. 
  \end{equation*}
  This inequality reflects the correct behavior in the weak coupling
  limit. Indeed, it is shown in \cite{K} that $-\Delta_\Neumann -
  \alpha V$ has at least one negative eigenvalue whenever
  $\int_{\Gamma} V(|x|) \, dx > 0$, and that for $\alpha$ sufficiently
  small this eigenvalue, say $\lambda_1(\alpha)$, is unique and
  satisfies 
  \begin{equation} \label{weak}
    -a_1\, \alpha^{\frac{2}{2-d}}\, \leq \, \lambda_1(\alpha)\, \leq \,
    -a_2\, \alpha^{\frac{2}{2-d}},
    \quad \alpha\to 0,
  \end{equation}
  for suitable constants $a_1\geq a_2>0$ depending on $V$. This fact
  shows also that \eqref{LT-basic} does \emph{not} hold for $1\leq
  d<2$, $a\geq 0$ and $\gamma<(1+a)\frac{2-d}{2d}$. We do not know
  whether \eqref{LT-basic} holds in the endpoint case
  $\gamma=(1+a)\frac{2-d}{2d}$ when $1\leq d<2$ and $a>d-1$. 

  Similarly, when $\Gamma$ has global dimension $d=2$, one can show
  that $-\Delta_\Neumann - \alpha V$ has at least one negative
  eigenvalue whenever $\int_{\Gamma} V(|x|) \, dx > 0$. Hence
  \eqref{LT-basic} does \emph{not} hold for $d=2$, $a\geq 0$ and
  $\gamma=0$.  
\end{remark}

\begin{remark}[\textbf{Dirac-potential limit}]
  As we have seen in the previous remark, the condition
  $\gamma>(1+a)(2-d)/(2d)$ in Part (2) of Theorem \ref{th:mainTree}
  comes from the weak coupling limit. Now we explain that the
  condition $\gamma\geq(1-a)/2$ in Part (1) comes from what may be
  called the Dirac-potential limit. Consider the sequence of
  potentials $V_n = n \chi_{(0,n^{-1})}$. Using a trial function
  supported near the root $o$ one easily proves that $\tr
  (-\Delta_\Neumann - V_n)_-^{\gamma}$ is bounded away from zero
  uniformly in $n$. On the other hand, $\int V_n^{\gamma + \frac{a +
      1}{2}} g_0^{\frac a{(d-1)}} \, dx$ tends to zero if
  $\gamma<(1-a)/2$. This shows that the condition $\gamma\geq
  \frac{1-a}2$ is necessary for the validity of \eqref{LT-basic}. 
\end{remark}

\begin{remark}[\textbf{Slowly decaying potentials}]
  Assume that $V$ is a symmetric function which is locally
  sufficiently regular and obtains the asymptotics $V(t)\sim \alpha
  t^{-s}$ as $t\to\infty$ for some $s>0$, $\alpha>0$. By standard
  methods (see, e.g., \cite[Thm. XIII.6]{RS4}) one shows that the
  operator $-\Delta_\Neumann - V$ has only a finite number of negative
  eigenvalues provided $s>2$. However, the semi-classical expression
  for the number of negative eigenvalues, i.e. the right hand side of
  \eqref{strong} with $\gamma=0$, is only finite under the more
  restrictive condition $s>2d$. Our Corollary \ref{clrcor} with
  sufficiently large $a$ gives a quantitative estimate on the number
  of negative eigenvalues for the whole range of exponents $s>2$ if
  $d>2$. Similarly, in the case $1\leq d\leq 2$ we obtain quantitative
  information about the magnitude of the eigenvalues, which goes
  beyond semi-classics. 
\end{remark}

\begin{remark}[\textbf{Dirichlet boundary conditions}]
  The reader might wonder how our main theorems change, if a Dirichlet
  instead of a Neumann boundary condition is imposed at the root. Let
  $-\Delta_\Dirichlet$ be the self-adjoint operator in $L_2(\Gamma)$
  generated by the quadratic form \eqref{eq:kinetic} with form domain
  $H^1_0(\Gamma):=\{ \phi\in H^1(\Gamma): \ \phi(0)=0 \}$. By the
  variational principle, any bound for $-\Delta_\Neumann-V$ implies a
  bound for $-\Delta_\Dirichlet-V$. However, it turns out that
  inequalities for the latter operator hold for a strictly larger
  range of parameters. Indeed, the analog of Theorems \ref{finite-rh}
  states that the inequality 
 \begin{equation*}
   \tr (-\Delta_\Dirichlet - V)_-^{\gamma} \leq C(\gamma,a,\Gamma)
   \int_{\Gamma} V(|x|)_+^{\gamma + \frac{1 + a}{2}}
   g_0(|x|)^{\frac{a}{d-1}} \, dx. 
 \end{equation*}
 holds provided either $0\leq a<1$ and $\gamma\geq (1-a)/2$, or else
 $a\geq 1$ and $\gamma\geq 0$ and $d\neq 2$, or else $a\geq 1$ and
 $\gamma> 0$ and $d=2$. This follows (except for the statement for
 $\gamma=0$, $1\leq d<2$) from Theorem \ref{mainA0dirichlet}. There is
 also an analog of Theorem \ref{finite-rh} for $-\Delta_\Dirichlet$
 which is obtained by simply interchanging the two intervals of
 integration in the assumption \eqref{eq:clrass}. We omit the
 details. For spectral asymptotics of the operator $-\Delta_\Dirichlet
 - V$ we refer to \cite{NS2}. 
\end{remark}


\subsection{One-dimensional Schr\"odinger operators with
  metric}\label{sec:oned} 

Our symmetry assumptions will allow us to reduce the spectral
analysis of the operator $-\Delta_\Neumann - V$ to the spectral
analysis of a family of one-dimensional Schr\"odinger-type
operators. The main ingredient in the proof of
Theorem~\ref{th:mainTree} will be an inequality for such operators,
which is of independent interest.

We consider a positive, measurable and locally bounded function $g$ on
$[0,\infty)$ and denote by $H^1(\R_+,g)$ the space of all functions
$f\in H^1_\loc(\R_+)$ such that  
\begin{align*}
\int_0^{\infty} \left(|f'(t)|^2 + |f(t)|^2 \right) g(t) \, dt < \infty.
\end{align*}
The quadratic form
\begin{equation}\label{eq:unperturbed}
\int_0^\infty |f'(t)|^2 g(t)\,dt
\end{equation}
with form domain $H^1(\R_+,g)$ defines a self-adjoint operator $A_g$
in $L_2(\R_+,g)$. Note that this operator corresponds to the
differential expression 
$$
A_g=  -g^{-1}\frac{d}{dt}g\frac{d}{dt}\, ,
$$ 
and that
functions $f$ in its domain satisfy Neumann boundary conditions
$f'(0)=0$ at the origin (at least when $g$ is sufficiently regular near $0$).

For our first results we assume that $g$ grows sufficiently fast in
the sense that 
\begin{equation}\label{eq:infinitelength}
  \int_t^\infty \frac{ds}{g(s)}< \infty \qquad \forall\, t >0.
\end{equation}
We shall prove that under this condition the number of negative
eigenvalues of the Schr\"o\-din\-ger operators $A_g-V$ can be estimated in
terms of weighted $L_p$-norms of $V$. More precisely, one has

\begin{theorem}\label{clr}
  Assume \eqref{eq:infinitelength} and let $w:\R_+\to\R_+$ be a
  positive function such that for some $2<q\leq\infty$
  \begin{equation}\label{eq:clrass1d}
    M :=  \sup_{t\geq 0} \left(\int_0^t g(s)^{\frac q2} w(s)^{-\frac
    {q-2}2} \,ds\right)^{2/q}
    \int_t^\infty \frac{ds}{g(s)}  <\infty \ .
  \end{equation}
  Let $p:= q/(q-2)$. Then the inequality
  \begin{equation}\label{eq:clr1d}
    N(A_g-V) \leq C_p(w,g) \int_0^\infty V^p_+ w \,dt
  \end{equation}
  holds for all $V$, and the sharp constant $C_p(w,g)$ in
  \eqref{eq:clr1d} satisfies
  \begin{equation*}
    M^p \leq C_p(w,g) 
    \leq \left(1+p'\right)^{p-1} \left(1+\frac 1 {p'} \right)^p M^p.
  \end{equation*}
  Moreover, if $M=\infty$ then there is no constant $C_p(w,g)$ such
  that \eqref{eq:clr1d} holds for all $V$. 
\end{theorem}

By definition, if $q=\infty$ condition \eqref{eq:clrass1d} is
understood as 
\begin{equation*}
M:=\sup_{t\geq0} \left(\sup_{0\leq s\leq t} \,\frac{g(s)}{w(s)}\right)
\int_t^\infty \frac{ds}{g(s)}  <\infty,
\end{equation*}
and the \emph{sharp} constant is $C_1(w,g)=M$. This leads to the following
beautiful estimate. 

\begin{example}\label{ex:clrsharp}
  Taking $w(t)=g(t)\int_t^\infty g^{-1}(s)\,ds$ and $q=\infty$ one obtains
  \begin{equation} \label{ex:sharp}
    N(A_g-V) 
    \leq \int_0^\infty V(t)_+ \,g(t) \left(\int_t^\infty
      \frac{ds}{g(s)}\right) \,dt\, , 
  \end{equation}
  which is sharp (meaning that the estimate is no longer true for all
  $g$ and all $V$ if the right hand side is multiplied by a constant
  less than one). As a consequence one also finds 
  \begin{equation*}
    N(A_g-V) 
    \leq \int_0^\infty \frac{dt}{g} \int_0^\infty V_+ g \,dt\, .
  \end{equation*}
\end{example}

Theorem \ref{clr} gives a complete characterization of weights for
which the number of negative eigenvalues can be estimated by a
weighted norm of the potential. When $g$ grows very fast, the operator
$A_g$ will be positive definite and in this case one may not only ask
for the number of eigenvalues of $A_g-V$ below $0$ but also below the
bottom of the spectrum of $A_g$. We turn to this question next. We
assume, in addition to \eqref{eq:infinitelength}, that 
\begin{equation}
  \label{eq:posdef1d}
  \sup_{t> 0} \int_0^t g(s)\,ds \int_t^\infty \frac{ds}{g(s)} < \infty.
\end{equation}
This condition is necessary and sufficient for the
operator $A_g$ to be positive definite, see Proposition \ref{mazya}
below or \cite[Thm. 5.2]{S}. We 
denote the bottom of its spectrum by $\lambda(A_g)>0$ and assume that
$\lambda(A_g)$ is \emph{not} an eigenvalue of $A_g$. Let $\omega$ be the
unique (up to a constant) distributional solution of the differential equation 
\begin{equation}\label{eq:gs}
  -(g \omega')' = \lambda(A_g)\, g\, \omega \qquad \text{on} \ \R_+
\end{equation}
satisfying the boundary condition $\omega'(0)=0$. Since $\lambda(A_g)$
is not an eigenvalue, the function $\omega$ is not square-integrable
with respect to the weight $g$. We quantify the growth of $\omega^2 g$ by
assuming that
\begin{equation}
  \label{eq:growth1d}
  \int_0^\infty \omega^{-2}  g^{-1}\,ds<\infty .
\end{equation}
Under these conditions one has

\begin{theorem} 
  \label{positive}
  Assume \eqref{eq:infinitelength}, \eqref{eq:posdef1d} and
  \eqref{eq:growth1d}. Let $w:\R_+\to\R_+$ be a positive function such
  that for some $2<q\leq\infty$ 
  \begin{equation*}
    M :=    \sup_{t>0} \left(\int_0^t \omega^q g^{\frac q2} w^{-\frac
    {q-2}2} \,ds\right)^{2/q} 
    \int_t^\infty \omega^{-2} g^{-1}\,ds < \infty,
  \end{equation*}
  and put $p := \frac{q}{q-2}$. Then the inequality
  \begin{equation}
    N(A_g-V-\lambda(A_g)) \leq C_p(w,g,\omega) \int_0^\infty V^p_+ w\,dt
  \end{equation}
  holds for all $V$, and the sharp constant $C_p(w,g,\omega)$ satisfies
  \begin{equation}
    M^p \leq C_p(w,g,\omega) 
    \leq \left(1+p'\right)^{p-1} \left(1+\frac 1 {p'} \right)^p M^p.
  \end{equation}
\end{theorem}

Finally, we present some estimates without imposing the condition
\eqref{eq:infinitelength}. It is easy to see that if the integral in
\eqref{eq:infinitelength} is infinite, then $A_g-V$ will have a
negative eigenvalue for any non-negative $V\not\equiv 0$, hence no
estimate on the number of eigenvalues in terms of norms of $V$ can
hold. Below we shall prove that estimates on \emph{moments} of
eigenvalues do hold. For the sake of simplicity we restrict ourselves
to the case where $g$ has power-like growth, i.e.,  
\begin{align}\label{eq:gpower}
    0 < c_1:= \inf_{t>0} \frac{g(t)}{(1+t)^{d-1}} \leq 
    \sup_{t>0} \frac{g(t)}{(1+t)^{d-1}} =: c_2 <\infty
\end{align}
for some $d\geq 1$. Note that \eqref{eq:infinitelength} holds iff $d>2$. We shall consider inequalities of the form
\begin{equation} \label{LT-individual}
\tr (A_g - V)_-^{\gamma} \leq L \int_0^{\infty}
V(t)_+^{\gamma + \frac{a + 1}{2}} (1 + t)^a \, dt,
\qquad L=L(\gamma,a,d,c_1,c_2).
\end{equation}
In Remark \ref{exprelation} below we show that the relation between
the exponent of $V$ and that of the weight $(1+t)$ can not be
improved. Our result is 

\begin{theorem}
  \label{mainA0}
  Assume \eqref{eq:gpower} for some $d\geq 1$.
  \begin{enumerate}
  \item
    \label{it:smalla}
    Let either $1\leq d <2$ and $0\leq a\leq d-1$, or else $d\geq 2$
    and $0\leq a<1$. Then \eqref{LT-individual} holds iff $\gamma\geq
    (1+a)/2$. 
  \item
    \label{it:largea}
    Let either $1\leq d <2$ and $a> d-1$, or else $d=2$ and $a\geq
    1$. Then \eqref{LT-individual} holds iff $\gamma>
    (1+a)(2-d)/(2d)$. 
  \item
    \label{it:clr}
    Let $d> 2$ and $a\geq 1$. Then \eqref{LT-individual} holds for any
    $\gamma\geq 0$. 
  \end{enumerate}
\end{theorem}

Part \eqref{it:clr} is of course a consequence of Theorem \ref{clr}
(for $\gamma=0$) and of an argument by Aizenman and Lieb \cite{AL}
(for $\gamma>0$). Note carefully that for small $a$ (Part
\eqref{it:smalla}) the inequality \eqref{LT-individual} holds in the
endpoint case, while it does \emph{not} for large $a$ (Part
\eqref{it:largea}). This is a phenomenon due to the Neumann boundary
conditions which is not present when Dirichlet boundary conditions are
imposed instead, see Theorem \ref{mainA0dirichlet}. 

\subsection{Outline of the paper}

This paper is organized as follows. In Section~\ref{sec:general} we
prove Theorem \ref{nonsym} and a weighted version of it about
arbitrary, not necessarily regular, metric trees. In Section
\ref{sec:regular} we show how our main results, Theorems
\ref{finite-rh}, \ref{th:mainTree} and \ref{homo}, follow from the
results about one-dimensional Schr\"odinger operators in Subsection
\ref{sec:oned}. 
In Section \ref{sec:clr} we give the proofs of Theorems \ref{clr} and
\ref{positive}. Section \ref{sec:gn} is of auxiliary character and
contains the proof of a family of Sobolev interpolation inequalities
which will be useful in the proof of Theorem \ref{mainA0}. Finally, in
Section \ref{sec:moments} we will use a duality argument and estimates
for Dirichlet eigenvalues in order to obtain the statements of Theorem
\ref{mainA0}. 


\section{Eigenvalue estimates on general metric trees}
\label{sec:general}

This section is devoted to the proof of Theorem
\ref{nonsym}. Moreover, we shall also prove the following 
weighted analog.

\begin{theorem}
  \label{nonsymweight}
  Let $a>0$ and $\gamma >(1+a)/2$. Then there exists a constant
  $C_a(\gamma)$ such that
  \begin{align} 
    \label{eq:nonsymweight}
    \tr (-\Delta_\Neumann - V)_-^\gamma \leq C_a(\gamma)\,
    \int_{\Gamma} V(x)_+^{\gamma + \frac {1+a}2} |x|^a \, dx.
  \end{align}
\end{theorem}

We emphasize that the constant in \eqref{eq:nonsymweight} can be
chosen independently of the tree. For the proofs of Theorems
\ref{nonsym} and \ref{nonsymweight} we use the following results about
half-line operators.

\begin{proposition}
  \label{ek}
  Let $\Gamma=\R_+$ and $a\geq0$. Let $\gamma>(1+a)/2$ if $a>0$ and
  $\gamma\geq 1/2$ if $a=0$. Then there exists a constant
  $L_{\gamma,a}^{EK}$ such that 
  \begin{equation}
    \label{eq:ek}
    \tr\left(-\Delta_\Neumann-V \right)_-^\gamma
    \leq
    L_{\gamma,a}^{EK} \int_0^\infty V(t)^{\gamma+\frac{1+a}2}_+ t^a\,dt
  \end{equation}
  for all $V$.
\end{proposition}

To prove \eqref{eq:ek} we extend $V$ to an even function $W$ on $\R$.
Then the left hand side of \eqref{eq:ek} can be estimated from above
by the corresponding moments of the whole-line operator $-d^2/dx^2
-W$, and the claimed inequality for that operator follows from
\cite{EK} and \cite{W}. Using in addition the sharp constants from
\cite{HLT} and \cite{AL} one obtains for $a=0$ the following bounds on
the constants,
\begin{equation}
  \label{eq:eklt}
  L_{\gamma,0}^{EK} \leq 4\,  L_{\gamma,1}^{\cl}\
  \text{if}\ \gamma\geq \frac12,
  \qquad
  L_{\gamma,0}^{EK} \leq 2\,  L_{\gamma,1}^{\cl}\
  \text{if}\ \gamma\geq \frac32
\end{equation}
with $L_{\gamma,1}^\cl$ from \eqref{eq:lclass}. Note that the
inequality \eqref{eq:ek} with this constant for $\gamma=1/2$ and $a=0$
is sharp, and therefore so is \eqref{eq:nonsym} for $\gamma=1/2$. Now
we turn to the  

\begin{proof}[Proof of Theorems \ref{nonsym} and \ref{nonsymweight}]
  The idea is to impose Neumann boundary condition at all but one
  emanating edges of all vertices. This decreases the operator
  $-\Delta_\Neumann - V$. The resulting operator can be identified
  with a direct sum of half-line operators for which one can use
  Proposition~\ref{ek}.

  To be more precise, we decompose the graph $\Gamma=\bigcup_j
  \Gamma_j$ into a disjoint union of infinite halflines $\Gamma_j$.
  Then $L_2(\Gamma)=\bigoplus_j L_2(\Gamma_j)$ and $H^1(\Gamma)\subset
  \sum_j H^1(\Gamma_j)$. By the variational principle, this implies
  \begin{equation*}
    -\Delta_\Neumann - V 
    \geq \bigoplus_j \left( -\Delta_\Neumann^{\Gamma_j} - V_j \right),
  \end{equation*}
  where $-\Delta_\Neumann^{\Gamma_j}$ is the Neumann Laplacian on
  $\Gamma_j$ and $V_j$ is the restriction of $V$ to $\Gamma_j$. Hence
  Proposition \ref{ek} yields
  \begin{align*}
    \tr (-\Delta_\Neumann - V)_-^\gamma & \leq \sum_j
    \tr_{L_2(\Gamma_j)} \left( -\Delta_\Neumann^{\Gamma_j} - V_j \right)_-^\gamma \\
    & \leq L_{\gamma,\alpha}^{EK}\, \sum_j \int_{\Gamma_j}
    V_j(x)^{\gamma +\frac{1+a}2}_+ \dist(x,\partial\Gamma_j)^a \, dx \\
    & \leq L_{\gamma,\alpha}^{EK}\,\int_{\Gamma} V(x)_+^{\gamma +
      \frac {1+a}2} |x|^a\, dx,
  \end{align*}
  as claimed.
\end{proof}


\section{Eigenvalue estimates on regular trees}
\label{sec:regular}

In this section we show how our main results, Theorems
\ref{finite-rh}, \ref{th:mainTree} and \ref{homo}, can be deduced from
the results about one-dimensional Schr\"odinger operators in
Subsection \ref{sec:oned}. To do so, we exploit the symmetry of the
tree and the potential, which allows us to decompose $-\Delta_\Neumann
- V$ into a direct sum of half-line Schr\"odinger operators in
weighted $L_2$-spaces. We recall this construction next. 

\subsection{Orthogonal decomposition}
\label{sec:decomp}

In this subsection we recall the results of Carlson \cite{C} and of
Naimark and Solomyak \cite{NS2,NS1}. We need some notation. For each
$k \in \N$ we define the higher order branching functions $g_k:\R_+
\to \N_0$ by 
\begin{align*}
g_k(t) := \left\{
\begin{array}{l@{\quad}l}
 0, & t < t_k\, , \\
 1, & t_k \leq t <  t_{k+1}\, , \\
 b_{k+1}b_{k+2}\cdots b_n, & t_n \leq t< t_{n+1},\, k<n \, , 
\end{array}
\right.
\end{align*}
and introduce the weighted Sobolev space $\H^1_0((t_k,\infty), g_k)$
as the closure of $C_0^\infty(t_k,\infty)$ in the norm 
\begin{align*}
\left[ \int_{t_k}^\infty \left( |f'(t)|^2 + |f(t)|^2 \right) g_k(t)
  \,dt \right]^{\frac 12}.
\end{align*}
Let $\A_k$ be the self-adjoint operator  in $L_2((t_k,\infty),\, g_k)$
given by the quadratic form 
\begin{align*} 
\qf_k[f] := \int_{t_k}^{\infty} |f'(t)|^2 g_k(t) \, dt
\end{align*}
with form domain $\H^1_0((t_k,\infty), g_k)$.
Notice that the operators $\A_k$ with $k \geq 1$ satisfy
\emph{Dirichlet} boundary condition at $t_k$, while the operator
$\A_0$ satisfies \emph{Neumann} boundary condition at $t_0 = 0$. 

\noindent The following statement is taken from 
\cite{NS1} and \cite{S}.

\begin{proposition} \label{NStheorem}
Let $V\in L_\infty(\Gamma)$ be symmetric. Then $-\Delta_\Neumann -
V$ is unitarily equivalent to the orthogonal sum of operators 
  \begin{equation} \label{decomp}
    -\Delta_\Neumann - V \simeq (\A_0 - V) \oplus
    \sum_{k=1}^{\infty}\oplus \big(\A_k - V_k\big)^{[b_1...b_{k-1}(b_k-1)]}.
  \end{equation}
  Here the symbol $[b_1...b_{k-1}(b_k-1)]$ means that the operator $\A_k
  - V_k$ appears $b_1...b_{k-1}(b_k-1)$ times in the orthogonal sum, and
  $V_k$ denotes the restriction of $V$ to the interval $(t_k,\infty)$.
\end{proposition}


\subsection{Proof of Theorems \ref{finite-rh} and \ref{th:mainTree}}

Let us compare the operators $\A_k$ with each other.
From the definition of the function $g_k$ it follows that
\begin{align*}
\frac{\int_{t_k}^{\infty}\left(|f'|^2 - V_k |f|^2 \right)
g_k \, dt} {\int_{t_k}^{\infty} |f|^2 g_k \, dt} 
=
\frac{\int_{t_k}^{\infty}\left(|f'|^2 - V_k |f|^2 \right) 
g_0 \, dt} {\int_{t_k}^{\infty} |f|^2 g_0 \, dt}.
\end{align*}
Since every function $f \in \H^1_0((t_k,\infty),g_k)$ can be extended
by zero to a function in $\H^1(\R_+,g_0)$, the variational
principle shows that
\begin{align} \label{majortrace}
\tr (\A_k - V_k)_-^\gamma \leq \tr (\A_0 - \chi_{(t_k,\infty)}V)_-^\gamma
\end{align}
for any $k \in \N$ and $\gamma \geq 0$.

Assuming the validity of Theorems \ref{clr} and \ref{mainA0} we now give the 

\begin{proof}[Proof of Theorems \ref{finite-rh} and \ref{th:mainTree}]
  In the case of Theorem \ref{finite-rh} put $\gamma=0$ and let $q$
  and $w$ be such that \eqref{eq:clrass} holds. Moreover, put
  $p=q/(q-2)$. In the case of Theorem \ref{th:mainTree} let $\gamma$
  be as indicated there and put $p=\gamma+(1+a)/2$ and
  $w(t):=g_0(t)^{a/(d-1)}$. It follows from Theorems \ref{clr} and
  \ref{mainA0}, respectively, that in both cases there exists a
  constant $C$ such that 
  \begin{equation*}
    \tr (\A_0 - V)_-^\gamma
    \leq C \int_0^{\infty}\, V(t)_+^{p} w(t) \, dt  
  \end{equation*}
  for all $V$. Combining this with the orthogonal decomposition
  \eqref{decomp} and inequality \eqref{majortrace} we obtain 
  \begin{align*}
    \tr(-\Delta_\Neumann - V)_-^\gamma 
    =& \, \tr (\A_0 - V)_-^\gamma + \sum_{k=1}^\infty b_1\cdots
    b_{k-1}(b_k-1) \tr (\A_k - \chi_{(t_k,\infty)} V)_-^\gamma \\ 
    \leq& \, C \int_0^{\infty}\, V(t)_+^p w(t) \, dt \\
    & + C \sum_{k=1}^\infty \left(b_1\cdots b_{k-1}(b_k-1)
      \int_{t_k}^\infty V(t)_+^p w(t) \, dt
    \right) \\
    = & \, C \sum_{k=0}^{\infty} \int_{t_k}^{t_{k+1}} (b_0\cdots b_k)
    V(t)_+^p w(t) \, dt  \\
    = & \, C\, \int_\Gamma V(|x|)_+^p\,w(|x|) \, dx,
  \end{align*}
  as claimed.
\end{proof}


\subsection{Proof of Theorem \ref{homo}}\label{sec:homo}

In this subsection we assume that $g_0$ is the first branching
function of a homogeneous metric tree with edge length $1$ and
branching number $b>1$. Denote by $\lambda_b$ the bottom of its
essential spectrum and by $\omega$ the function on $\R_+$ satisfying
in distributional sense 
\begin{equation*}
-(g_0 \omega')' = \lambda_b g_0\, \omega\, ,
\end{equation*}
\begin{equation*}
\omega'(0)=0, \quad \omega(j+)= \omega(j-),\quad
\omega'(j-)=b\omega'(j+), \quad j\in\N\, . 
\end{equation*}
In the proof of Theorem \ref{homo} we need the following technical result.

\begin{lemma}
  \label{efgrowth}
  There exist constants $0<C_1<C_2<\infty$ such that
  \begin{equation}
    \label{eq:efgrowth}
    C_1\, \frac{1+t}{\sqrt{g_0(t)}} \leq \omega(t)
    \leq C_2\, \frac{1+t}{\sqrt{g_0(t)}}\, , 
    \quad t\geq 0\, .
  \end{equation}
\end{lemma}

Assuming this for the moment we give the

\begin{proof}[Proof of Theorem \ref{homo}]
  Proceeding in the same way as in the proof of Theorems
  \ref{finite-rh} and \ref{th:mainTree} one sees that it suffices to
  prove that 
  \begin{equation}
    \label{eq:homoproof}
    N(\A_0 -V -\lambda_b) \leq C \int_0^{\infty}\, V(t)_+^{p} w(t) \, dt\, .
  \end{equation}
  We shall deduce this from Theorem \ref{positive} with $g=g_0$. By
  the explicit form of $g_0$ we see that \eqref{eq:infinitelength} and
  \eqref{eq:posdef1d} are satisfied. Moreover,
  $\lambda_b=\lambda(\A_0)$ and $\omega$ is the generalized ground
  state of $\A_0$ in the sense of \eqref{eq:gs}. It follows from Lemma
  \ref{efgrowth} that the assumption \eqref{eq:growth1d} is satisfied
  and that one has 
  \begin{align*}
    & \left(\int_0^t \omega^q g_0^{\frac q2} w^{-\frac {q-2}2}
    \,ds\right)^{2/q} \int_t^\infty \omega^{-2} g^{-1}\,ds\\
    & \qquad \leq \left(\frac{C_2}{C_1}\right)^2 
    \left(\int_0^t (1+s)^q w^{-\frac{q-2}2} \,ds\right)^{2/q} \frac1{1+t}.
  \end{align*}
  Hence \eqref{eq:homoproof} follows from Theorem \ref{positive}.
\end{proof}

We are left with the

\begin{proof}[Proof of Lemma \ref{efgrowth}]
  A direct calculation shows that
  \begin{equation*}
    \omega(t)= \alpha_j\, \cos(\mu(t-j))+\beta_j\,
    \cos(\mu(j+1-t)),\quad j < t < j+1,
  \end{equation*}
  with $\mu := \sqrt{\lambda_b}$, $\alpha_0:=1,\, \beta_0:=0$ and
  \begin{align*}
    \alpha_{j-1} \cos\mu + \beta_{j-1} = \alpha_j+\beta_j\cos\mu\, , 
    \quad -\alpha_{j-1} = b\, \beta_j \, .
  \end{align*}
  This can be rewritten as
  \begin{equation*}
    \left(\begin{array}{c}
        \alpha_j \\
        \beta_j
      \end{array} \right)     b^{-\frac 12}\, \left(\begin{array}{cc}
        2 &  b^{\frac 12}\\
        -b^{- \frac 12} & 0
      \end{array} \right)
    \left(\begin{array}{c}
        \alpha_{j-1} \\
        \beta_{j-1}
      \end{array} \right)\, ,
  \end{equation*}
  and by induction one easily finds that
  \begin{equation*}
    \left(\begin{array}{c}
        \alpha_j \\
        \beta_j
      \end{array} \right)     b^{- \frac j2}\, \left(\begin{array}{cc}
        j+1 & j\, b^{\frac 12} \\
        -j\, b^{-\frac 12} & -j+1
      \end{array} \right)\,
    \left(\begin{array}{c}
        \alpha_0 \\
        \beta_0
      \end{array} \right)\, .
  \end{equation*}
  This implies
  \begin{equation*}
    \omega(t) = g_0(t)^{-\frac 12} (j+1) 
    \left(\cos(\mu(t-j))- \frac{j}{j+1} b^{-\frac 12}
    \cos(\mu(j+1-t)) \right)
  \end{equation*}
  if $j<t<j+1$, and hence
  \begin{equation} \label{asymp}
    \omega(t) \, \sim \, g_0(t)^{-\frac 12} \, (1+t)\, \varphi(t),
    \quad t\to \infty,
  \end{equation}
  where $\varphi$ is periodic with period $1$ and
  $$
  \varphi(t)= \cos \mu t - b^{-\frac 12}\cos(\mu(1-t)), \quad 0<t<1.
  $$
  The estimates
  $$
  \frac{b^{\frac 12}-b^{-\frac 12}}{b^{\frac 12}+b^{-\frac 12}} \,
  \geq \, \varphi(t)\, 
  \geq\,  b^{-\frac 12}\, \frac{b^{\frac 12} - b^{-\frac 12}}{b^{\frac
  12}+b^{-\frac 12}}
  >0,\quad 0<t<1\, ,
  $$
  and the asymptotics \eqref{asymp} imply that \eqref{eq:efgrowth}
  holds for all sufficiently large $t$. On the other hand, by the
  Sturm oscillation theorem (or by direct calculation) $\omega$ is
  bounded and bounded away from zero on compacts. This proves the
  lemma. 
\end{proof}

\section{Estimates on the number of eigenvalues}
\label{sec:clr}


\subsection{Proof of Theorem \ref{clr}}

Our goal in this section is to prove the statements of Theorem
\ref{clr}. An important ingredient will be weighted Hardy-Sobolev
inequalities. The characterization of all admissible weights is
independently due to Bradley, Maz'ya and Kokilashvili. The constant in
\eqref{eq:mbest} below is due to Opic. We refer to \cite[Thm. 6.2]{OK}
for the proof and further historical remarks.

\begin{proposition}
  \label{mazya}
  Let $2\leq q\leq\infty$. The inequality
  \begin{equation}\label{eq:m}
    \left(\int_0^\infty |w(r)u(r)|^q \,dr \right)^{2/q}
    \leq S^2 \int_0^\infty |v(r)u'(r)|^2 \,dr
  \end{equation}
  holds for all absolutely continuous functions $u$ on $[0,\infty)$
  with $\lim_{r\to\infty}u(r)=0$ if and only if
  \begin{equation}\label{eq:mcrit}
    T:= \sup_{r>0} \left(\int_0^r |w(s)|^q \, ds\right)^{1/q}
    \left(\int_r^\infty |v(s)|^{-2}\,ds\right)^{1/2} <\infty.
  \end{equation}
  In this case, the sharp constant $S$ in \eqref{eq:m} satisfies
  \begin{equation}\label{eq:mbest}
    T \leq S 
    \leq \left(1+\frac q2\right)^{1/q} \left(1+\frac 2q\right)^{1/2} T.
  \end{equation}
\end{proposition}

If $q=\infty$, then \eqref{eq:mcrit} means
\begin{equation*}
  T:= \sup_{r>0} \left(\sup_{0\leq s\leq r} |w(s)| \right)
  \left(\int_r^\infty |v(s)|^{-2}\,ds\right)^{1/2} <\infty,
\end{equation*}
and in \eqref{eq:mbest} one has $T=S$. Now everything is in place to give the

\begin{proof}[Proof of Theorem \ref{clr}]
  Let $w\geq 0$ such that $M$ defined in \eqref{eq:clrass1d} is
  finite. Then Proposition \ref{mazya} yields for all $u\in
  H^1(\R_+,g)$,
  \begin{equation}
    \label{aux} 
    \left(\int_0^\infty |u|^q g^{\frac q2} w^{-\frac{q-2}2} \,dt\right)^{2/q} 
    \leq S^2 \int_0^\infty |u'|^2 g\,dt,
  \end{equation}
  where
  \begin{equation*}
    M \leq S^2
    \leq \left(1+\frac q2\right)^{2/q} \left(1+\frac 2q\right) M.
  \end{equation*}
  We now use an argument in the spirit of \cite{GGMT} to deduce
  \eqref{eq:clr1d} from \eqref{aux}. Let $\omega$ be the solution of
  $-(g\omega')'-V\omega g=0$ that satisfies the boundary condition
  $\omega'(0)=0$. By Sturm-Liouville theory (see, e.g.,
  \cite[Thm. 14.2]{Wm}) the number of zeros of $\omega$
  coincides with the number $N$ of negative eigenvalues of $A_g -
  V$. Denote these 
  zeros by $0<a_1<a_2<\ldots<a_N<\infty$ and apply \eqref{aux} to $u
  \omega \chi_{(a_j,a_{j+1})}$. Integrating by parts and using 
  H\"older's inequality (noting that $1/p+2/q=1$) we obtain
  \begin{align*}
    & \left(\int_{a_j}^{a_{j+1}} |\omega|^q g^{\frac q2}
      w^{-\frac{q-2}2} \, dt\right)^{2/q} \leq S^2
    \int_{a_j}^{a_{j+1}} |\omega'|^2 g \, dt
    = S^2 \int_{a_j}^{a_{j+1}} V |\omega|^2 g \, dt \\
    &\qquad \leq S^2 \left(\int_{a_j}^{a_{j+1}} V^pw\, dt\right)^{1/p}
    \left(\int_{a_j}^{a_{j+1}} |\omega|^q\, g^{\frac q2}
      w^{-\frac{q-2}{2}}\, dt \right)^{2/q}\, .
  \end{align*}
  This implies that
  $$
  1 \leq S^{2p}\, \int_{a_j}^{a_{j+1}} V^pw\, dt\, , \quad \forall\,
  j=1,\dots N\, .
  $$
  Summing this inequality over all intervals $(a_j,a_{j+1})$ we obtain
  \begin{equation*}
    N(A_g - V) \leq S^{2p} \int_0^\infty V^p_+ w\,dt.
  \end{equation*}
  This proves \eqref{eq:clr1d} and shows that the sharp constant
  satisfies $C(w)\leq S^{2p}$. The lower bound $C(w)\geq S^{2p}$
  follows from Theorem \ref{lowest} below. This implies also that
  \eqref{eq:clr1d} does not hold if $M=\infty$ and 
  completes the proof.
\end{proof}

For later reference we include

\begin{example}
  Assume that $g$ satisfies \eqref{eq:gpower}
  for some $d>2$. Then for any $1\leq a<\infty$
  \begin{equation*}
    N(A_g-V) \leq C_a \int_0^\infty V^{\frac {1+a}2}_+ (1+t)^a \,dt
  \end{equation*}
  where
  \begin{equation*}
    \left(\frac{c_1}{c_2}\right)^{\frac{1+a}2} M_a^{\frac{1+a}2}
    \leq C_a
    \leq \frac{(2a)^a}{(a+1)^{\frac{a+1}2}(a-1)^{\frac{a-1}2}}
    \left(\frac{c_2}{c_1}\right)^{\frac{1+a}2} M_a^{\frac{1+a}2}\, . 
  \end{equation*}
  and
  \begin{align*} 
    M_a & := \sup_{t>0} \left(\int_0^t
      (1+s)^{\frac{(d-1)(a+1)-2a}{a-1}} \,ds\right)^{\frac{a-1}{a+1}}
    \int_t^\infty (1+s)^{-d+1}\,ds \\
    & = \left(\frac{a-1}{a+1}\right)^{\frac{a-1}{a+1}}
    (d-2)^{-\frac{2a}{a+1}}.
  \end{align*}
  (For $a=1$ one has $(c_1/c_2) M_1 \leq C_1 \leq (c_2/c_1) M_1$ and
  $M_1:=(d-2)^{-1}$.) This follows by choosing $w(t)=(1+t)^a$ and
  $q=2(a+1)/(a-1)$ after elementary calculations. 
\end{example}

It is also illustrative to include another proof of estimate
\eqref{ex:sharp} in Example~\ref{ex:clrsharp}: The Birman-Schwinger
principle implies
\begin{equation} 
  \label{bs} 
  N(A_g-V) \, \leq \, \tr_{L^2(\R_+,g dt)}\,
  \left(V^{\frac 12}_+\, A_g^{-1}\, V^{\frac 12}_+\right).
\end{equation} 
Since the operator $V^{\frac 12}_+\, A_g^{-1}\, V^{\frac 12}_+$ is
non-negative, we have
\begin{equation} 
  \label{tr}
  \tr_{L^2(\R_+,g dt)}\, \left(V^{\frac 12}_+\, A^{-1}\,  V^{\frac
      12}_+\right) = \int_0^\infty\, G(t,t)\, V(t)_+ \, g(t)\, dt,
\end{equation}
where $G(t,t)$ is the diagonal of the Green function of the operator
$A$. It follows from Sturm-Liouville theory (see, e.g.,
\cite[Thm. 7.8]{Wm}) that 
\begin{equation*}
  G(t,t) = \frac{u_1(t)\, u_2(t)}{g(t) W(t)}\, ,  
\end{equation*} 
where $u_1,u_2$ are two linearly independent solutiuons of $-(gu')'=0$
and $W=u_1'u_2 -u_1 u_2'$ is their Wronskian. A direct calculation
gives 
$$
u_1(t) = 1,\quad u_2(t)= \int_t^\infty\, \frac{ds}{g(s)} \, , \quad W(t)
=\frac{1}{g(t)}\, .
$$
In view of \eqref{bs} and \eqref{tr} this yields estimate
\eqref{ex:sharp}.


\subsection{Proof of Theorem \ref{positive}}

In this subsection we are working under the assumptions
\eqref{eq:infinitelength}, \eqref{eq:posdef1d} and \eqref{eq:growth1d}
of Theorem \ref{positive}. Recall that $\omega$ is the `ground state'
of the operator $A$. Since $g$ may be non-smooth (it is a step
function in the case of the tree) the differential equation
\eqref{eq:gs} has to be understood in quadratic form sense, i.e., 
\begin{equation}\label{eq:gs2} 
\int_0^\infty \omega' f' g\,dt = \lambda(A) \int_0^\infty \omega f
g\, dt
\end{equation}
for all $f\in H^1(\R_+,g)$ with compact support in $[0,\infty)$. The
following identity is usually called ground state representation. 

\begin{lemma}\label{gsr}
For any $h = \omega^{-1} f \in \omega^{-1} H^1(\R_+,g)$,
\begin{equation}\label{eq:gsr}
\int_0^\infty |f'|^2 g\,dt - \lambda(A) \int_0^\infty |f|^2 g\,dt
= \int_0^\infty |h'|^2 \omega ^2 g\,dt.
\end{equation}
\end{lemma}

We include a sketch of the proof for the sake of completeness.

\begin{proof}
It suffices to consider $h\in C_0^\infty(\overline{\R_+})$. Then
\begin{equation*}
|(\omega h)'|^2 = \omega^2 |h'|^2 + \omega' (\omega |h|^2)'
\end{equation*}
and \eqref{eq:gsr} follows from \eqref{eq:gs2} with $f=\omega |h|^2$.
\end{proof}

With \eqref{eq:gsr} at hand we can proceed to the

\begin{proof}[Proof of Theorem \ref{positive}]
We denote by $B$ the operator in $L_2(\R_+,\omega^2 g)$ corresponding
to the quadratic form 
\begin{equation*}
\int_0^\infty |h'|^2 \omega ^2 g\,dt
\end{equation*}
with form domain $H^1(\R_+,\omega^2 g)$. Then by the ground state
representation~\eqref{gsr} and Glazman's lemma (see
e.g. \cite[Thm. 10.2.3]{BS}) 
\begin{equation*}
N(A-V-\lambda(A)) = N(B-V),
\end{equation*}
and the result follows from Theorem \ref{clr}.
\end{proof}


\section{Sobolev interpolation inequalities}
\label{sec:gn}



In this section we fix a parameter $d\geq1$ and study inequalities of
the form 
\begin{align} \label{eq:gn1}
\left( \int |u|^q (1+t)^{\beta q-1}\,dt\right)^{2/q}
 \leq & \, \, K(q,\beta,d) \left( \int |u'|^2 (1+t)^{d-1}
   \,dt\right)^{\theta} \\
& \left( \int |u|^2 (1+t)^{d-1} \,dt\right)^{1-\theta} \nonumber
\end{align}
for all $u \in H^1(\R_+,(1+t)^{d-1})$. We are interested in the values
of $\beta$ and $q$ for which this inequality holds. We always fix
\begin{equation}\label{eq:theta}
  \theta:=\frac{d-2\beta}2.
\end{equation}
In the endpoint case $q=\infty$ we use the convention that
\eqref{eq:gn1} means
\begin{equation*}
\sup |u|^2 (1+t)^{2\beta}
\leq K(\infty,\beta,d) \left( \int |u'|^2 (1+t)^{d-1} \,dt\right)^{\theta}
\left( \int |u|^2 (1+t)^{d-1} \,dt\right)^{1-\theta}
\end{equation*}
for all $u\in H^1(\R_+,(1+t)^{d-1})$. Note that this makes sense even
in the special case $\beta=0$ (where the product $\beta q$ in
\eqref{eq:gn1} is not well-defined).

\begin{theorem}\label{gn1}
Let $d\geq 1$ and $\frac{d-2}2\leq\beta\leq\frac d2$.
\begin{enumerate}
\item\label{it:gnbetasmall}
If $1<d\leq 2$ and $0<\beta\leq\frac{d-1}2$, or if $d> 2$ and
$\frac{d-2}2\leq\beta\leq\frac{d-1}2$, then \eqref{eq:gn1} holds for
all $2\leq q \leq \infty$.
\item\label{it:gnbetalarge}
If $d\geq 1$ and $\frac{d-1}2<\beta\leq\frac{d}2$, then \eqref{eq:gn1}
holds for all $2\leq q \leq \left(\beta-\frac{d-1}2\right)^{-1}$.
\item\label{it:gnbeta0}
If $1\leq d<2$ and $\beta=0$, then \eqref{eq:gn1} holds for $q=\infty$.
\item\label{it:gnbetaneg}
If $1\leq d\leq 2$ and $-\frac{2-d}2\leq\beta\leq 0$, then
\eqref{eq:gn1} does \emph{not} hold for $2\leq q < \infty$.
\item\label{it:gnbetaneginfty}
If $1\leq d< 2$ and $-\frac{2-d}2\leq\beta< 0$, or if $d=2$ and
$\beta= 0$, then \eqref{eq:gn1} does \emph{not} hold for $q = \infty$.
\item\label{it:gnbetalargeqlarge}
If $d\geq 1$ and $\frac{d-1}2<\beta\leq\frac{d}2$, then \eqref{eq:gn1}
does \emph{not} hold for
$\left(\beta-\frac{d-1}2\right)^{-1}<q\leq\infty$. 
\end{enumerate}
\end{theorem}

We refer to Figure 1 below for the region of allowed parameters.

\begin{remark}\label{gnexprelation}
  In \eqref{eq:gn1} the exponent $\beta q-1$ of the weight on the left
  hand side is coupled to the interpolation exponent $\theta$ in
  \eqref{eq:theta}. This is in a certain sense optimal. Indeed, if the
  inequality 
  \begin{align*}
    \left( \int |u|^q (1+t)^{\sigma-1}\,dt\right)^{2/q}
    \leq & \, \, K \left( \int |u'|^2 (1+t)^{d-1}
      \,dt\right)^{\theta}\\
    & \left( \int |u|^2 (1+t)^{d-1} \,dt\right)^{1-\theta} \nonumber
  \end{align*}
  holds for some $\sigma>0$ and all $u\in H^1(\R_+,(1+t)^{d-1})$, then
  necessarily $\sigma\leq q(d-2\theta)/2$. (To see this put
  $u(t)=v(lt)$ and let $l\to 0$.) Note that with the value
  \eqref{eq:theta} of $\theta$ one has $q(d-2\theta)/2=\beta q$. 
\end{remark}

We break the proof into several lemmas which prove inequality
\eqref{eq:gn1} in the endpoint cases.

\begin{lemma}\label{gn2}
If $1< d\leq 2$ and $0<\beta\leq\frac{d-1}2$, or if $d> 2$ and
$\frac{d-2}2\leq\beta\leq\frac{d-1}2$, then \eqref{eq:gn1} holds for
$q=2$ with the constant
\begin{equation*}
K(2,\beta,d) = \beta^{-d+2\beta}.
\end{equation*}
\end{lemma}

\begin{proof}
Integration by parts shows
\begin{align*}
\int |u|^2 (1+t)^{2\beta-1}\,dt
& = (-\beta)^{-1} \Re \int \overline u u'\left((1+t)^{2\beta}-1\right)\,dt \\
& \leq \beta^{-1} \int |u| |u'| (1+t)^{2\beta} \,dt.
\end{align*}
We shall assume now that $\beta<\frac{d-1}2$. The proof in the case of
equality follows along the same lines. Then
$p:=\frac{d-2\beta}{d-1-2\beta}$ satisfies $1<p<\infty$, and by
H\"older we can continue to estimate
\begin{align*}
& \int |u|^2 (1+t)^{2\beta-1}\,dt
\leq \beta^{-1} \left(\int |u|^2 (1+t)^{2\beta-1} \,dt\right)^{1/p} \\
& \qquad \times \left(\int |u|^{\frac{p-2}{p-1}} |u'|^{\frac p{p-1}}
  (1+t)^{\frac{2\beta(p-1)+1}{p-1}} \,dt \right)^{\frac{p-1}p}.
\end{align*}
By the definition of $p$ one has
\begin{equation*}
\frac{2\beta(p-1)+1}{p-1}= \frac{(d-1)(p-2)}{2(p-1)} + \frac{(d-1)p}{2(p-1)},
\end{equation*}
and hence again by H\"older,
\begin{align*}
& \int |u|^{\frac{p-2}{p-1}} |u'|^{\frac p{p-1}}
(1+t)^{\frac{2\beta(p-1)+1}{p-1}} \,dt \\
& \qquad \leq \left( \int |u|^{2} (1+t)^{(d-1)} \,dt \right)^{\frac{p-2}{2(p-1)}}
\left( \int |u'|^{2} (1+t)^{(d-1)} \,dt \right)^{\frac p{2(p-1)}}.
\end{align*}
This proves the inequality with the claimed constant.
\end{proof}

\begin{lemma}\label{gninfty}
If $1<d\leq 2$ and $0<\beta\leq\frac{d-1}2$, or if $d> 2$ and
$\frac{d-2}2\leq\beta\leq\frac{d-1}2$, then \eqref{eq:gn1} holds for
$q=\infty$ with the constant
\begin{align*}
K(\infty,\beta,d) = \left(\frac2{d-2\beta}\right)^{d-2\beta}
\left(\frac{d-1-2\beta}{2\beta}\right)^{d-1-2\beta}.
\end{align*}
\end{lemma}

Here we use the convention that
$0^0=1$. Hence for $\beta=\frac{d-1}2$ one has
$K(\infty,\frac{d-1}2,d) = 2$.

\begin{proof}
Let $p:=\frac2{d-2\beta}$. Our assumptions imply that $\frac2d<p\leq
2$ if $1<d\leq 2$ and $1\leq p\leq 2$ if $d> 2$.
By Schwarz we estimate
\begin{align*}
|u(t)|^p & \leq p \int_t^\infty |u|^{p-1} |u'| \,ds \\
& \leq p \left(\int_0^\infty |u'|^2 (1+s)^{d-1} \,ds\right)^{1/2}
\left(\int_t^\infty |u|^{2(p-1)} (1+s)^{-d+1} \,ds\right)^{1/2}
\end{align*}
This proves the assertion if $p=1$, i.e., $\beta=\frac{d-2}2$ and
$d>2$. If $p=2$ the assertion follows from the estimate
\begin{align*}
\int_t^\infty |u|^{2(p-1)} (1+s)^{-d+1} \,ds
\leq (1+t)^{-2(d-1)} \int_0^\infty |u|^{2(p-1)} (1+s)^{d-1} \,ds.
\end{align*}
In the remaining case $1<p<2$ we use H\"older to obtain
\begin{align*}
& \int_t^\infty |u|^{2(p-1)} (1+s)^{-d+1} \,ds \\
& \qquad \leq \left(\int_t^\infty (1+s)^{-\frac{(d-1)p}{2-p}} \,ds\right)^{2-p}
\left(\int_0^\infty |u|^{2} (1+s)^{d-1} \,ds\right)^{p-1} \\
& \qquad = \left(\frac{2-p}{dp-2}\right)^{2-p} (1+t)^{-dp+2}
\left(\int_0^\infty |u|^{2} (1+s)^{d-1} \,ds\right)^{p-1}.
\end{align*}
This proves the inequality with the claimed constant.
\end{proof}

\begin{lemma}\label{gnbeta0}
If $1\leq d<2$ and $\beta=0$, then \eqref{eq:gn1} holds for $q=\infty$
with the constant
\begin{align*}
K(\infty,0,d) = (2d)^d (2(d-1))^{-2(d-1)} (2-d)^{-1}.
\end{align*}
\end{lemma}

\begin{proof}
If $d=1$ one has
\begin{equation}\label{eq:agmon}
|u(t)|^2 \leq 2 \int_t^\infty |u| |u'|\,ds
\leq 2 \left(\int_0^\infty |u|^2 \,ds \right)^{1/2}
\left(\int_0^\infty |u'|^2\,ds\right)^{1/2},
\end{equation}
as claimed. If $1<d<2$ then we estimate for any $R>0$
\begin{align*}
|u(t)|^2
& \leq 2 \left( \int_0^R |u| |u'|\,ds + \int_R^\infty |u| |u'|\,ds \right) \\
& \leq 2 \left( \left(\int_0^\infty |u'|^2 s^{d-1} \,ds \right)^{1/2}
  \|u\|_\infty \left(\int_0^R s^{-d+1} \,ds \right)^{1/2} \right. \\
& \qquad\qquad \left.
+ \left(\int_0^\infty |u'|^2 s^{d-1} \,ds \right)^{1/2}
\left(\int_0^\infty |u|^2 s^{d-1} \,ds \right)^{1/2} R^{-d+1} \right)
\\
& = 2 \left(\int_0^\infty |u'|^2 s^{d-1} \,ds \right)^{1/2}
\Bigg [ \|u\|_\infty (2-d)^{-1/2} R^{(2-d)/2} \\
& \qquad \qquad \qquad   \qquad \qquad \qquad  + \left(\int_0^\infty
    |u|^2 s^{d-1} \,ds \right)^{1/2} R^{-d+1} \Bigg ].v
\end{align*}
Choosing $t$ such that $u(t)=\|u\|_\infty$ and optimizing with respect
to $R$ we find that
\begin{equation*}
\|u\|_\infty^2 \leq K \left(\int |u'|^2 s^{d-1} \,ds \right)^{d/2}
\left(\int |u|^2 s^{d-1} \,ds \right)^{(2-d)/2}
\end{equation*}
with the constant as claimed. This implies (and, by a scaling
argument, is actually equivalent to) the assertion.
\end{proof}

\begin{lemma}\label{gnd1}
If $d=1$ and $0<\beta\leq\frac{1}2$, then \eqref{eq:gn1} holds for
$q=2$ with the constant
\begin{equation*}
K(2,\beta,1) = 2^{-2\beta}(1-2\beta)^{2\beta-1}\beta^{-1}.
\end{equation*}
\end{lemma}

\begin{proof}
It suffices to prove the inequality
\begin{equation*}
\int |v|^2 s^{-1+2\beta}\,ds
\leq K \left( \int |v'|^2 \,ds\right)^{(1-2\beta)/2}
\left( \int |v|^2 s^{d-1} \,ds\right)^{(1+2\beta)/2}.
\end{equation*}
(Actually, a scaling argument as in the proof of Theorem \ref{gn1}
below shows that this inequality is equivalent -- with the same
constant -- to the inequality \eqref{eq:gn1}.) Using \eqref{eq:agmon}
we estimate for any $R>0$
\begin{align*}
\int |v|^2 s^{-1+2\beta}\,ds
& \leq \|v\|_\infty^2 \int_0^R s^{-1+2\beta}\,ds + \|v\|_2^2 R^{-1+2\beta} \\
& \leq \beta^{-1} \|v\| \|v'\| R^{2\beta} + \|v\|_2^2 R^{-1+2\beta},
\end{align*}
and the claim follows by optimizing with respect to $R$.
\end{proof}

\begin{proof}[Proof of Theorem \ref{gn1}]
First assume that $1<d\leq 2$ and $0<\beta\leq\frac{d-1}2$, or $d> 2$
and $\frac{d-2}2\leq\beta\leq\frac{d-1}2$. The assertion
\eqref{it:gnbetasmall} has been proved in the endpoint cases $q=2$ and
$q=\infty$ in Lemmas \ref{gn2} and \ref{gninfty}. Estimating
\begin{equation*}
\int |u|^q (1+t)^{\beta q-1}\,dt
\leq \sup \left(|u|^{q-2}(1+t)^{\beta (q-2)}\right) \int |u|^2
(1+t)^{\beta 2-1}\,dt
\end{equation*}
we obtain the assertion \eqref{it:gnbetasmall} also in the case $2<q<\infty$.

Next we prove the assertion \eqref{it:gnbetalarge}. Let $d\geq 1$,
$\frac{d-1}2<\beta\leq\frac{d}2$. First assume that $q=2$. If $d=1$,
the inequality holds by Lemma \ref{gnd1}. If $d>1$ we put $p :(2\beta-d+1)^{-1}$ and apply H\"older's inequality to find
\begin{equation*}
\int |u|^2 (1+t)^{2\beta -1}\,dt
\leq \left( \int |u|^2 (1+t)^{d-2}\,dt\right)^{\frac{p-1}p}
\leq \left(\int |u|^2 (1+t)^{d-1}\,dt\right)^{\frac1p}.
\end{equation*}
Estimating the first factor on the right side using Lemma \ref{gn2}
with $\beta\leq\frac{d-1}2$ we obtain the assertion in the case
$q=2$. Now let $q= \left(\beta-\frac{d-1}2\right)^{-1}$. We estimate
\begin{equation*}
\int |u|^q (1+t)^{2\beta -1}\,dt
\leq \left( \sup |u|^2(1+t)^{d-1} \right)^{\frac{d-2\beta}{2\beta-d+1}}
\left(\int |u|^2 (1+t)^{d-1}\,dt\right).
\end{equation*}
The first factor on the right side is estimated using \eqref{eq:agmon}
if $d=1$ and using Lemma \ref{gninfty} with $\beta\leq\frac{d-1}2$ if
$d>1$. This proves the assertion in the case $q\left(\beta-\frac{d-1}2\right)^{-1}$. By H\"older's inequality we
obtain \eqref{it:gnbetalarge} for arbitrary
$2<q<\left(\beta-\frac{d-1}2\right)^{-1}$.

The assertion \eqref{it:gnbeta0} was proved in Lemma \ref{gnbeta0}.

To prove the negative results let $1\leq d\leq 2$ and assume that \eqref{eq:gn1} holds for some $\beta$ and some $2\leq q\leq\infty$. We apply the inequality to
the function $u(t)=v(t/l)$, where $v$ is a smooth function with
bounded support. Letting $l\to\infty$ we obtain
\begin{equation}\label{eq:gnhom}
\left( \int |v|^q s^{\beta q-1}\,ds\right)^{2/q}
\leq K(q,\beta,d) \left( \int |v'|^2 s^{d-1} \,ds\right)^{\theta}
\left( \int |v|^2 s^{d-1} \,ds\right)^{1-\theta}.
\end{equation}
Note that $v$ can be chosen non-zero in a neighborhood of the
origin. We deduce that the inequality can not hold for $\beta<0$, and
if $q<\infty$ then it can not hold for $\beta=0$ either. This proves
assertion \eqref{it:gnbetaneg} and the first part of
\eqref{it:gnbetaneginfty}. It remains to prove that \eqref{eq:gn1} or
equivalently \eqref{eq:gnhom} does not hold if $d=2$, $\beta=0$ and
$q=\infty$. This follows by considering the sequence of trial
functions $v_n(s):=\min\{1,(\log n -\log s)/\log n\}$ if $s\leq n$ and
$v_n(s)=0$ for $s> n$.

Finally, to prove \eqref{it:gnbetalargeqlarge} let $d\geq 1$ and $\frac{d-1}2<\beta\leq\frac{d}2$. Again we apply the inequality to the function $u(t)=v(t/l)$, where $v$ is a smooth function with bounded support. As $l\to 0$, the left hand side decays like $l^{2/q}$ (resp. becomes constant when $q=\infty$) whereas the right hand side decays like $l^{2\beta-d+1}$. We conclude that the condition $q\leq \left(\beta-\frac{d-1}2\right)^{-1}$ is necessary for \eqref{eq:gn1} to hold.
\end{proof}


\section{Estimates for moments of eigenvalues}
\label{sec:moments}


Our goal in this section will be to prove the Lieb-Thirring bounds in Theorem \ref{mainA0}. Throughout we will assume that $g$ has power-like growth in the sense of \eqref{eq:gpower} for some $d\geq 1$.


\subsection{One-bound-state inequalities and duality}
\label{sec:duality}

A first step towards Theorem \ref{mainA0} is to prove that the lowest eigenvalue of the operator $A_g-V$ can be estimated from below by a weighted $L_p$-norm of the potential.

\begin{theorem} \label{lowest}
Assume \eqref{eq:gpower} for some $d\geq 1$ and let $a, \gamma \geq 0$. Then the inequality
\begin{equation} \label{eq:lowest}
\sup \spec \left( (A_g - V)_-^{\gamma} \right) \leq C
\int_{\R_+} V(t)_+^{\gamma + \frac{a + 1}{2}} g(t)^{\frac a{d-1}} \, dt,
\qquad C=C(\gamma,a,d,c_1,c_2),
\end{equation}
holds for all $V$ if and only if $a$ and $\gamma$ satisfy the assumptions of Theorem~\ref{mainA0}.
\end{theorem}

In the case $\gamma=0$, inequality \eqref{eq:lowest} means that if $\int_{\R_+}
V(t)_+^{\frac{a + 1}{2}} (1+t)^a \, dt < C^{-1}$ then
$\inf\spec (\A_g - V)\geq 0$.

The proof of Theorem \ref{lowest} is based on the following abstract duality result, which does not use the explicit form of $g$.

\begin{proposition}\label{duality}
Assume that the parameters $a>-1$, $\gamma\geq 0$ and
$p:=\gamma+\frac{1+a}2$ are related to the parameters $2<q\leq\infty$,
$\frac{d-2}2\leq\beta<\frac d2$ and $\theta:=\frac{d-2\beta}2$ by
\begin{equation}\label{eq:duality}
p=\frac q{q-2}, \qquad q=\frac {2p}{p-1}, \qquad
a=\frac{(d-1-2\beta)q+2}{q-2}, \qquad
\beta =\frac{dp-1-a}{2p},
\end{equation}
see Figure 1.
Then the inequality \eqref{eq:lowest} holds if and only if
\begin{equation}\label{eq:dualitygn}
\left( \int |u|^q g^{\frac{\beta q-1}{d-1}}\,dt\right)^{2/q}
\leq K(q,\beta,g) \left( \int |u'|^2 g \,dt\right)^{\theta} \left(
  \int |u|^2 g \,dt\right)^{1-\theta} \,.
\end{equation}
for all $u\in H^1(\R_+,g)$. In this case, the constants are related by
\begin{equation}\label{eq:dualityconst}
K(q,\beta,g) = L^\frac {q-2}q \theta^{-\theta} (1 - \theta)^{\theta - 1}
\end{equation}
\end{proposition}

In the case $q=\infty$, \eqref{eq:dualitygn} means
\begin{equation*}
\sup |u|^2 g^{\frac{2 \beta}{d-1}}
\leq L \theta^{-\theta} (1 - \theta)^{\theta - 1} \left( \int |u'|^2 g
  \,dt\right)^{\theta}
\left( \int |u|^2 g \,dt\right)^{1-\theta}\, .
\end{equation*}
for all $u\in H^1(\R_+,g)$.

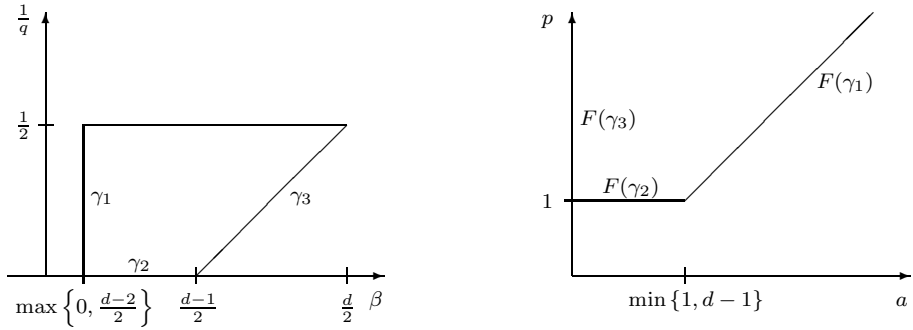
\begin{figure} \label{fig1}
\scriptsize
\begin{center}
\setlength{\unitlength}{5mm}
\begin{picture}(24,9)(0,1)
\put(1,2){\vector(0,1){7}}
\put(0,2){\vector(1,0){10}}
\put(0.2,8.7){$\frac 1q$}
\put(9.6,1.2){$\beta$}

\put(0.8,6){\line(1,0){0.4}}
\put(0.2,5.85){$\frac 12$}

\put(2,1.8){\line(0,1){4.2}}
\put(2,6){\line(1,0){7}}
\put(9,6){\line(-1,-1){4}}

\put(5,1.8){\line(0,1){0.4}}
\put(9,1.8){\line(0,1){0.4}}

\put(0.2,1){$\max\left\{0, \frac {d-2}2\right\}$}
\put(4.5,1){$\frac {d-1}2$}
\put(8.8,1){$\frac d2$}

\put(2.2,4){$\gamma_1$}
\put(3.2,2.2){$\gamma_2$}
\put(7.5,4){$\gamma_3$}


\put(15,2){\vector(0,1){7}}
\put(15,2){\vector(1,0){9}}
\put(14.2,8.7){$p$}
\put(23.6,1.2){$a$}

\put(14.8,4){\line(1,0){3.2}}
\put(18,4){\line(1,1){5}}

\put(18,1.8){\line(0,1){0.4}}
\put(16.5,1.2){$\min \left\{ 1, d-1\right\}$}
\put(14.2,3.8){$1$}

\put(15.2,6){$F(\gamma_3)$}
\put(15.8,4.2){$F(\gamma_2)$}
\put(21.5,7){$F(\gamma_1)$}
\end{picture}
\caption{Parameter range of the Sobolev interpolation
  inequalities. Here $F\left( 1/q, \beta \right) = ( q,
  (d-1-2\beta)q+2)/(q-2)$ and $F(\gamma_1) = \{ (p,a) : p =
  (a+1)/\min\{2,d\}\}$.}
\end{center}
\end{figure}

\begin{proof}[Proof of Proposition \ref{duality}]
Below we will only consider $u \in H^1(\R_+,g)$ and $V \geq 0$ such that
the right hand side of \eqref{eq:lowest} is finite.

Equation \eqref{eq:lowest} holds for all $V$ if and only if
\begin{align}\label{eq:prop421}
\frac {\int |u'|^2 g \, dt - \int V |u|^2 g \, dt} {\int |u|^2 g \,
  dt} \geq -\left( L \int V^p g^{\frac a{d-1}} \, dt \right)^{2/(2p-1-a)}
\end{align}
holds for all $u$ and $V$. Write $V = \alpha W$ with $\alpha$ such that
\begin{align}\label{eq:prop422}
\int W^p g^{\frac a{d-1}} \, dt = 1.
\end{align}
Thus \eqref{eq:prop421} holds for all $u$ and $V$ if and only if
\begin{align} \label{eq:prop424}
\sup_{\alpha > 0} \left( \alpha \int W |u|^2 g \, dt - \alpha^{\frac
  {1}{1-\theta}} L^{\frac {q-2}{q(1-\theta)}} \int |u|^2 g \, dt \right) \leq
  \int |u'|^2 g \, dt
\end{align}
holds for all $u$ and all $W$ obeying \eqref{eq:prop422}. By calculating
the supremum we find that \eqref{eq:prop423} holds for all $u$ and all $W$ obeying \eqref{eq:prop422} if and only if
\begin{align}\label{eq:prop423}
\sup \left\{ \int W |u|^2 g \, dt: \int W^p g^{\frac a{d-1}} \, dt = 1 \right\}
\leq K \left( \int |u'|^2 g \, dt\right)^\theta 
\left( \int |u|^2 g \, dt\right)^{1 - \theta}
\end{align}
for all $u$. By duality
\begin{align*}
\sup \left\{ \int W |u|^2 g \, dt: \ \int W^p g^{\frac a{d-1}} \, dt = 1 \right\}
= \left( \int |u|^q g^{\frac {\beta q - 1}{d - 1} } \, dt \right)^{2/q}.
\end{align*}
Hence \eqref{eq:prop423} holds for all $u$ if and only if
\eqref{eq:dualitygn} holds for all $u$.
\end{proof}

\begin{proof}[Proof of Theorem \ref{lowest}]
Assumption \eqref{eq:gpower} implies that Theorem \ref{gn1} holds (with another constant) if $(1+t)^{d-1}$ is replaced by $g$. Simple arithmetic shows that if $(q,\beta)$ and $(p,a)$ are related as in \eqref{eq:duality}, then the allowed values $(q,\beta)$ in Theorem \ref{gn1} correspond to the allowed values $(p,a)$ in Theorem~\ref{mainA0}. In view of Proposition \ref{duality} we obtain the assertion of Theorem~\ref{lowest}.
\end{proof}

\begin{remark}\label{exprelation}
  We claim that if the inequality
  \begin{equation}\label{eq:exprelation}
    \sup \spec \left( (A_g - V)_-^{\gamma} \right) \leq C
    \int_{\R_+} V(t)_+^{\gamma + \frac{1 + a}{2}} g(t)^b \, dt
  \end{equation}
  holds for some $\gamma\geq 0$, $a\geq 0$, $b\geq 0$ and all $V$, then one has necessarily $b\geq a/(d-1)$. Obviously, the inequality becomes weaker as $b$ increases. This motivates why we restrict ourselves to the case $b=a/(d-1)$ when considering the inequalities \eqref{LT-individual}.

  To prove the claim we apply a similar duality argument as in the proof of Proposition \ref{duality} and find that \eqref{eq:exprelation} is equivalent to
  \begin{equation*}
    \left( \int |u|^q g^{\frac{p-b}{p-1}}\,dt \right)^{2/q}
    \leq K \left( \int |u'|^2 g\,dt \right)^\theta 
    \left( \int |u|^2 g\,dt \right)^{1-\theta}\!\!,
    \quad u\in H^1(\R_+,g),
  \end{equation*}
  where $p$ and $q$ are as in that proposition and $\theta=(p-\gamma)/p$. It follows from Remark \ref{gnexprelation} that $(d-1)(p-b)/(p-1) + 1 \leq q(d-2\theta)/2$. This means $b\geq a/(d-1)$, as claimed.
\end{remark}

\subsection{Estimates in the case of a Dirichlet boundary condition}\label{sec:dirichlet}

Here we will establish the analog of Theorem \ref{mainA0} when a Dirichlet instead of a Neumann boundary condition is imposed at the origin. More precisely we denote by $A_\Dirichlet$ the self-adjoint operator in $L_2(\R_+,g)$ corresponding to the quadratic form \eqref{eq:unperturbed} with form domain $H^1_0 (\R_+,g):=\{f\in H^1(\R_+,g):\ f(0)=0\}$. In this case the conditions for the validity of a Lieb-Thirring inequality become much simpler than in Theorem \ref{mainA0}.

\begin{theorem} \label{mainA0dirichlet}
Assume \eqref{eq:gpower} for some $d\geq 1$ and let $a\geq 0$, $\gamma>0$. Then the inequality
\begin{equation} \label{LT-individualdirichlet}
\tr (A_\Dirichlet - V)_-^{\gamma} 
\leq L \int_{\R_+} V(t)_+^{\gamma + \frac{a + 1}{2}} (1 + t)^a \, dt,
\qquad L=L(\gamma,a,d,c_1,c_2),
\end{equation}
holds for all $V$ if and only if $a, \gamma$ satisfy
\begin{align*}
\gamma\geq\frac{1-a}2 & \qquad\text{if}\ 0\leq a<1,\\
\gamma>0 & \qquad\text{if}\  a\geq 1.
\end{align*}
\end{theorem}

We emphasize that we did not discuss the case $\gamma=0$ in Theorem
\ref{mainA0dirichlet}. 

When proving Theorem \ref{mainA0dirichlet} we will use a result from \cite{EF1} and \cite{EF2} concerning the operator $- \frac {d^2}{d r^2} - \frac{1}{4 r^2} - W$ in $L_2(\R_+)$ with a Dirichlet boundary condition at the origin.

\begin{proposition}\label{ef}
Let $0\leq a<1$ and $\gamma \geq \frac{1-a}2$ or $a \geq 1$ and
$\gamma > 0$, then
\begin{align} \label{eq:ef}
\tr \left(- \frac {d^2}{d r^2} - \frac{1}{4 r^2} - W \right)_-^\gamma
\leq C_{\gamma,a}
\int_{\R_+} W(r)^{\gamma + \frac{1+a}2}_+ r^a\, dr
\end{align}
with a constant $C_{\gamma,a}$ independent of $W$.
\end{proposition}

Before we can apply this estimate we have to replace the (possibly non-smooth) function $g$ by a smooth function with the same behavior at infinity. To this end we consider the self-adjoint operator $\B_\Dirichlet$ in $L_2(\R_+)$ corresponding to the quadratic form
\begin{align}
\begin{split}\label{quadraticformb}
\b_\Dirichlet [u] & = \int_{\R_+} \left| \left( \frac{u(t)} {(1 + t)^{(d-1)/2}}
  \right)'\right|^2 (1 + t)^{d-1} \, dt \\
& = \int_{\R_+} \left( |u'|^2 + \frac {(d - 1) (d - 3)
|u|^2} {4 (1 + t)^2}\right) \, dt
\end{split}
\end{align}
defined on $H^1_0(\R_+)$. We prove now that the eigenvalues of $A_\Dirichlet - V$ can be estimated -- modulo a change in the coupling constant -- from above and below by those of $\B_\Dirichlet - V$. A similar idea was used in \cite{FSW} to obtain Lieb-Thirring inequalities for Schr\"odinger operators with background potentials.

\begin{lemma} \label{auxiliary}
Assume \eqref{eq:gpower} for some $d\geq 1$ and put $\beta:=c_2/c_1$. Then for any $V\geq 0$ and $\gamma \geq 0$ we have
\begin{align} \label{eigenvalue-ineq}
\tr (\B_\Dirichlet - \beta^{-1} V )_-^\gamma
\leq \tr (A_\Dirichlet - V)_-^\gamma 
\leq \tr (\B_\Dirichlet - \beta V )_-^\gamma.
\end{align}
\end{lemma}

\begin{proof}
We shall prove that for any $\tau>0$
\begin{equation}\label{eq:number}
N(\B_\Dirichlet-\beta^{-1} V+\tau) \leq N(A_\Dirichlet- V+\tau) \leq N(\B_\Dirichlet-\beta V+\tau).
\end{equation}
This will imply the statement since
\begin{equation*}
\tr T_-^\gamma = \gamma \int_0^\infty \tau^{\gamma-1} N(T+\tau)\,d\tau.
\end{equation*}
To prove the second inequality in \eqref{eq:number} suppose that
$$
\int_{\R_+}\left(|f'|^2 - V |f|^2\right)\, g\, dt < -\tau \int_{\R_+}\, |f|^2 g\, dt
$$
for some $f\in H^1_0(\R_+,g)$. Using \eqref{eq:gpower} we conclude that
\begin{align*}
c_1 \int_{\R_+}\left(|f'|^2 - \beta V |f|^2\right)\, (1+t)^{d-1}\, dt & \, 
\leq \, \int_{\R_+}\left(|f'|^2 - V |f|^2\right) \, g\, dt \\
& \leq \, -\tau \int_{\R_+}\, |f|^2g\, dt \, \\
& \leq \, -\tau c_1 \int_{\R_+}\, |f|^2 (1+t)^{d-1} \, dt\, .
\end{align*}
It follows from Glazman's lemma (see, e.g., \cite[Thm. 10.2.3]{BS}) that
$$
N(A_\Dirichlet- V+\tau) \leq N(\tilde A_\Dirichlet-\beta V+\tau),
$$
where $\tilde A_\Dirichlet$ denotes the operator $L^2(\R_+,(1+t)^{d-1})$ corresponding to the quadratic form $\int |f'|^2 (1+t)^{d-1}\, dt$ with a Dirichlet boundary condition. Since $\tilde A_\Dirichlet-\beta V$ in $L^2(\R_+,(1+t)^{d-1})$ is unitarily equivalent to $\B_\Dirichlet-\beta V$ in $L^2(\R_+)$, we obtain the second inequality in \eqref{eq:number}. The first one is proved similarly.
\end{proof}

\begin{proof}[Proof of Theorem \ref{mainA0dirichlet}]
We may assume that $V\geq 0$. We use the operator inequality
$$
- \frac {d^2}{d r^2} - \frac{1}{4 r^2}
\leq - \frac {d^2}{d r^2} + \frac{(d-1)(d-3)}{4 r^2} \, .
$$
(Note also that the form domain of the operator on the LHS is strictly larger than $H^1_0(\R_+)$.) It follows that
$$
\tr (\B_\Dirichlet - \beta V )_-^\gamma \,
\leq \,\tr \left(- \frac {d^2}{d r^2} - \frac{1}{4 r^2} - \beta V \right)_-^\gamma.
$$
The result now follows from Proposition \ref{ef} and Lemma \ref{auxiliary}.
\end{proof}


\subsection{Putting it all together}

Finally we give the

\begin{proof}[Proof of Theorem \ref{mainA0}]
The variational principle implies that the eigenvalues of the Dirichlet and the Neumann problems interlace (see, e.g., \cite[Thm. 10.2.5]{BS}). Hence
\begin{align*}
\tr (A - V)_-^\gamma
\leq \sup\spec \left((A - V)_-^\gamma \right) + \tr(A_\Dirichlet - V)_-^{\gamma}.
\end{align*}
We estimate the first term on the right hand side via Theorem \ref{lowest} (recall \eqref{eq:gpower}) and the second one via Theorem \ref{mainA0dirichlet}. This completes the proof of the `if' part of the statement. The `only if' statement follows from the `only if' part of Theorem \ref{lowest}.
\end{proof}



\bibliographystyle{amsalpha}

\end{document}